\documentclass{article}
\usepackage[utf8]{inputenc}
\usepackage{graphicx,color}
\usepackage{amsmath,amssymb,amsfonts,amsthm,ascmac,bm}
\usepackage{algorithm,algorithmic,lscape}
\usepackage{enumerate}

\usepackage[numbers]{natbib}

\usepackage[margin=25mm]{geometry}
\usepackage{mathtools}
\mathtoolsset{showonlyrefs,showmanualtags}
\usepackage[titletoc,title]{appendix}

\newcommand{\argmin}{\mathop{\rm argmin}\limits}
\newcommand{\prox}{\mathrm{prox}}

\newtheorem{proposition}{Proposition}[section]
\newtheorem{theorem}{Theorem}[section]
\newtheorem{lemma}{Lemma}[section]

\newtheorem{assumption}{Assumption}[section]

\title{Exact Penalty Method for Knot Selection of B-Spline Regression}
\author{Shotaro Yagishita
\quad
Jun-ya Gotoh}
\date{\today}

\begin{document}
\maketitle
\begin{abstract}
This paper presents a new approach to selecting knots at the same time as estimating the B-spline regression model. 
Such simultaneous selection of knots and model is not trivial, but our strategy can make it possible by employing a nonconvex regularization on the least square method that is usually applied.  
More specifically, motivated by the constraint that directly designates (the upper bound of) the number of knots to be used, we present an (unconstrained) regularized least square reformulation, which is later shown to be equivalent to the motivating cardinality-constrained formulation. 
The obtained formulation is further modified so that we can employ a proximal gradient-type algorithm, known as GIST, for a class of non-smooth non-convex optimization problems. 
We show that under a mild technical assumption, the algorithm is shown to reach a local minimum of the problem. 
Since it is shown that any local minimum of the problem satisfies the cardinality constraint, the proposed algorithm can be used to obtain a spline regression model that depends only on a designated number of knots at most. 
Numerical experiments demonstrate how our approach performs on synthetic and real data sets.
\end{abstract}

\section{Introduction}
B-spline regression \cite{de1978practical} is a popular methodology for nonparametric estimation of a nonlinear model represented by a  spline function, which is defined by a weighted sum of piecewise polynomials called basis functions.
On the other hand, its high flexibility in describing the model's nonlinearity is (implicitly or explicitly) attributed to the high degree of freedom of the spline (e.g., the degree of the polynomials, and shapes and locations of the basis functions), so in practice, the user has to take a lot into account so as to capture the nonlinearity of the phenomenon of interest while avoiding overfitting to data. 

As in most practices, we suppose in this paper that the order of the polynomials is fixed at a single value (for example, the third order). 
To tame the degree of freedom of the spline (for the fixed order of polynomials), it is common to add regularizing terms to the sum of squared residuals (SSR) of data points from the model. 
For example, it is popular to suppress the integral of the second derivatives of the spline function (over an interval) \citep{o1986statistical}, or 
the second difference of the coefficients of three consecutive basis functions in addition to the SSR \citep{eilers1996flexible}.
Either regularizer mitigates the overfitting by restricting the curvature of the spline. 
One of the advantages of such regularizer is that it is represented by a convex quadratic term, which helps us avoid the ill condition of the objective function of the least squares estimation especially when the number of knots (or basis functions) is large. 

A more interesting and challenging way of suppressing the degree of freedom is selecting knots, which directly determine the locations and shapes of basis functions.  
Although fitting through the P-spline, for example, can be cast as an unconstrained convex quadratic optimization, resulting in solving a linear equation, it is not easy to introduce a tractable optimization criterion that simultaneously determines the locations of knots. (Note that in estimating the ordinary P-spline, the locations of knots are to be determined before solving the convex quadratic optimization.) 
In their preprint, \citet{goepp2018spline} recently presented an $\ell_0$-(pseudo) norm penalty on the vector of (specific order) difference of the coefficient vector of basis functions as a regularizer so as to select the locations of knots, and proposed a heuristics based on an approximation of the $\ell_0$-regularizer, which is called A-spline. 
Since it can be shown that the $\ell_0$-norm penalty represents the number of knots used, their method actually aims to find a subset of knots that simultaneously constrains the number of knots used and the model's SSR. 

Inspired by their approach, we propose a new methodology based on continuous optimization, starting with an intuitive $\ell_0$-constrained least square formulation where (the upper bound of) the number of knots to be used to define the spline is supposed to be designated by the user. 
Despite its similarity, our proposal is more intuitive and attractive than the A-spline in that i)
while the A-spline employs the cardinality as a penalty term in the objective and includes a less interpretable penalty parameter, the number of used knots is directly designated in our formulation, and ii) we further reformulate the $\ell_0$-constraint into an equivalent continuous regularizer, which is an example of the generalized trimmed lasso \citep{yagishita2022exact}. 
In particular, the equivalence of the continuous regularizer is established by an exact penalty result.
Namely, we show that if the penalty parameter of the regularizer is fixed at a large value, the unconstrained continuous optimization is equivalent to the cardinality constrained optimization. 
As a result, the hyper-parameter of our methodology virtually becomes the number of knots only, and the user only has to deal with the more interpretable parameter rather than the less interpretable penalty parameter employed in the A-spline.

To obtain a solution to our penalized formulation, we propose to apply the GIST \citep{wright2009sparse,gong2013general}, a proximal gradient type algorithm, to the further reduced formulation obtained by a change of variables. 
We show that under a mild assumption the sequence generated by the method converges to a local optimum, which satisfies the cardinality constraint. 
Therefore, with our formulation and the algorithm, the user can perform the simultaneous selection of locations of knots and least square fitting by specifying (the upper bound of) the number of used knots before running the algorithm.
Aside from dealing with the hyper-parameters, it may appear that knots' locations can be determined almost continuously by preparing many knots candidates for either our method or the A-spline.
However, our experiment reveals that unlike our method, the A-spline fails to estimate the model when a large number of knots candidates are available. 

Contributions of the paper are summarized as follows:
\begin{itemize}
    \item A new continuous optimization formulation and its solution method for the B-spline regression are proposed.
    Our methodology guarantees to obtain a spline regression model using no greater than the designated number of knots owing to the exact penalty property (Theorem \ref{thm:e.p.p.}) and the global convergence to a local optimal solution (Theorem \ref{thm:global-convergence}).
    We point out a mistake of \citet{goepp2018spline} on a fact that is also important in our formulation and prove the correct one (Theorem \ref{thm:active_knot_indicator}).
    \item Numerical experiments show that our proposal performs more stable estimation than the A-spline \cite{goepp2018spline}.
    In particular, our methodology can select knots from many candidates, namely, the locations of the knots can be selected almost continuously, but the A-spline cannot.
\end{itemize}

The rest of this paper is organized as follows.
In the next section, a brief review of B-spline regression and our proposed formulation are described.
Section \ref{sec:solution-method} presents a method to obtain a local optimal solution of our formulation satisfying a constraint on the number of knots.
Numerical examples to demonstrate how our proposed methodology works are reported in Section \ref{sec:experiments}.
Finally, Section \ref{sec:conclusion} concludes the paper.

\section{Formulation}\label{sec:formulation}
In this section, we first review the ordinary formulation of B-spline regression (see \citet{de1978practical} for details).
After giving a motivating example, we will present a modified formulation, which includes a variant of cardinality constraint, and our reformulation with an alternative penalty function.

\subsection{B-spline regression}
Let $(x_i,y_i)\in\mathbb{R}^2,~i=1,\dots,n$, be $n$ paired observations of explanatory (or independent) variable $x$ and a response (or dependent) variable $y$. 
Given such a data set, our aim is to estimate a non-linear model 
\begin{align}
y_i = s(x_i)+\epsilon_i, \quad i=1,\dots,n,
\end{align}
where $s:\mathbb{R}\to\mathbb{R}$ is a function and $\epsilon_i$ is the residual. 

For $s$, we consider a piece-wise polynomial function.
Let $p\in\mathbb{N}\cup\{0\}$ be fixed, and let $t_{-p},t_{-p+1},\dots,t_{l+p}$ be the entire $l+2p+1$ \emph{knots} involved in the regression.
We assume that $t_{-p}<\cdots<t_0<t_1<\cdots<t_{l-1}<t_l<\cdots<t_{l+p}$ holds, and that the observed values $x_1,\dots,x_n$ of the explanatory variable exist within the interval $[t_0,t_l)$. (As will be elaborated on later, the other $2p$ knots, $t_{-p},\dots,t_{-1}$ and $t_{l+1},\dots,t_{l+p}$, lie out of the interval; they are necessary for defining the basis functions.)  
Given those knots, the basis functions $B^{(p)}_{j}$ are computed in a recursive way, as follows. 
First, the $j$-th basis of order $0$ is defined by
\begin{align}
    B^{(0)}_{j}(x)=\bm{1}_{[t_j,t_{j+1})}(x)\coloneqq
    \begin{cases}
        1, & t_j\le x< t_{j+1},\\
        0, & \mbox{otherwise},\\
    \end{cases}
\end{align}
$j=-p,\dots,l+p-1$, and the basis functions of order $p\ge1$ are then defined, recursively, by the formula given by
\begin{align}
    B^{(q)}_{j}(x)=\frac{x-t_{j}}{t_{j+q}-t_{j}}B^{(q-1)}_{j}(x)+\frac{t_{j+q+1}-x}{t_{j+q+1}-t_{j+1}}B^{(q-1)}_{j+1}(x),\quad j=-p,\dots,l+p-q-1,
\end{align}
$q=1,\ldots,p$.
Figure \ref{fig:my_label} illustrates the structure of restrictions of basis functions to $[t_0,t_l)$.

\begin{figure}[ht]
    \centering
    \includegraphics[width=\columnwidth]{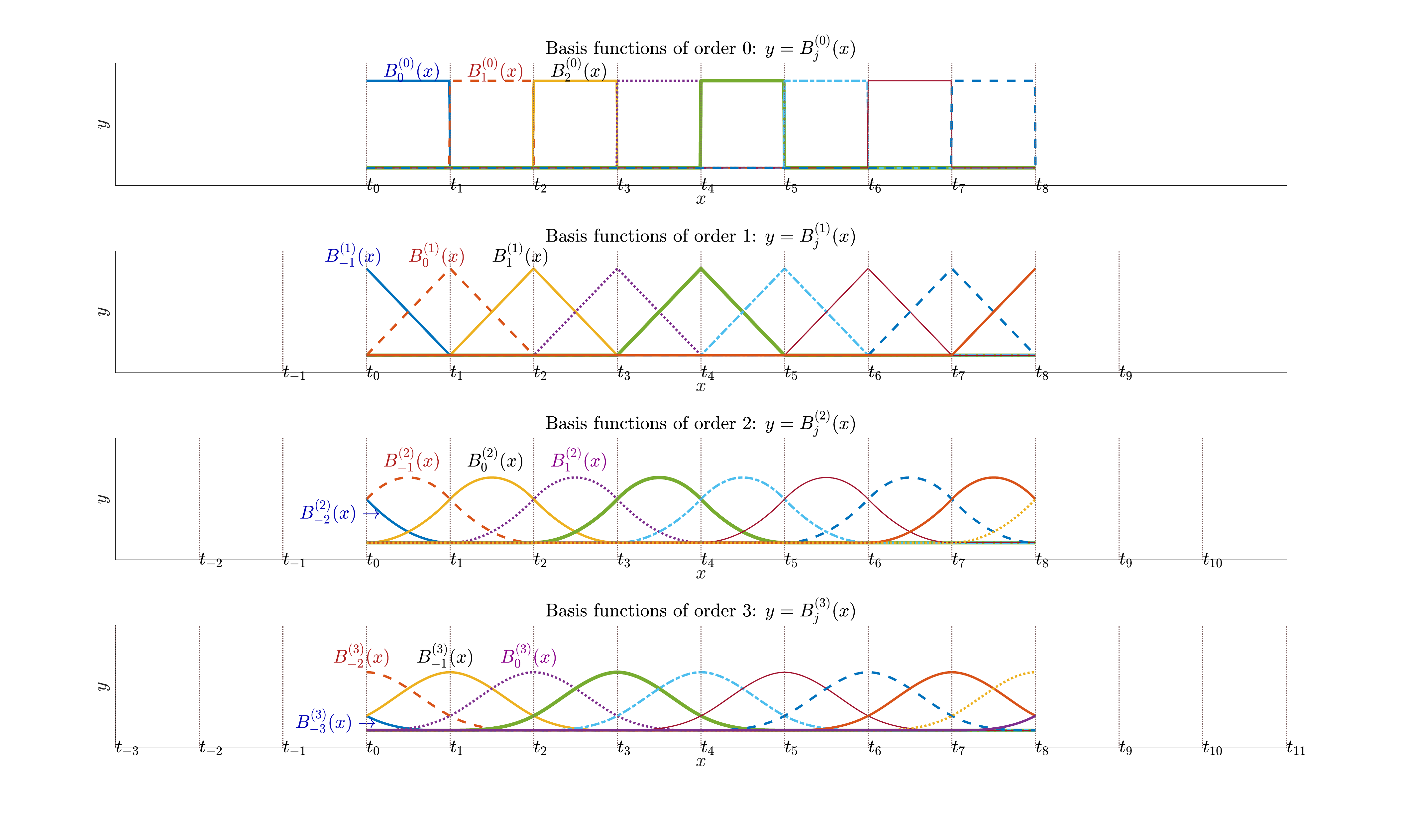}
    \caption{Basis functions defined with equally-spaced knots for $p=0,1,2,3$}
    \label{fig:my_label}
    \begin{quote}\footnotesize
        The above pictures demonstrate the case where $l=9$ knots are given, $t_0,t_1,\dots,t_8$. 
        The top row shows eight basis functions of order $0$, each being an indicator function of corresponding interval. 
        Starting with these, the higher order basis functions are defined recursively up to the designated degree, e.g., $p=3$ in this figure. 
        For example, in the second row, the basis function $B^{(1)}_0$, depicted as the dashed orange line, is computed from $B^{(0)}_0$ and $B^{(0)}_1$, which are in the top row; the basis function $B^{(2)}_0$, depicted as the solid yellow curve in the third row, is computed from $B^{(1)}_0$ and $B^{(1)}_1$, which are in the second row, and so on.  
    \end{quote}
\end{figure}

The B-spline function of order $p$ is then defined by
\begin{align}
    s(x)=\sum_{j=-p}^{l-1}\alpha_{j+p+1}B^{(p)}_j(x),
\end{align}
where $B^{(p)}_{j}$ denotes the $j$-th basis function of order $p$ and $\alpha_{j}$ is the coefficient of $B^{(p)}_{j}$, where $j=-p,-p+1,\dots,l-2,l-1$. 

By construction, the B-spline $s(x)$ is a piecewise polynomial of order $p$, and with the higher $p$, the B-spline $s$ can represent higher nonlinearity in $x$.
Although the order $p$ can be a hyper-parameter which enables the user to improve the fitting via calibration, we assume that $p$ is fixed throughout the paper so as to avoid unnecessary disturbance. 
In fact, the cubic spline (i.e., $p=3$) is known to be sufficient enough to minimize the integral of the curvature of $s$, and we may fix at $p=3$ (although our methodology is independent of the number).

Let $\bm{y}=(y_1,\dots,y_n)^\top\in\mathbb{R}^n$ and $\bm{\alpha}=(\alpha_1,\dots,\alpha_{l+p})^\top\in\mathbb{R}^n$, and let $\bm{B}^{(p)}$ be the $n\times(l-p)$ matrix of the values of the basis functions $B_j^{(p)}(x)$ at $x=x_{1},\dots,x_{n}$, namely,
\begin{align}
    \bm{B}^{(p)}\coloneqq
    \begin{pmatrix}
    B^{(p)}_{-p}(x_1) & \cdots & B^{(p)}_{l-1}(x_1)\\
    \vdots      &        & \vdots \\
    B^{(p)}_{-p}(x_n) & \cdots & B^{(p)}_{l-1}(x_n)\\
    \end{pmatrix}.
\end{align}
With this notation, (half) the sum of squared residuals (SSR) for the B-spline function becomes a convex quadratic function denoted by
\begin{align}
\frac{1}{2}\sum_{i=1}^{n}\epsilon_i^2=\frac{1}{2}\|\bm{y}-\bm{B}^{(p)}\bm{\alpha}\|_2^2.
\label{obj:ssr}
\end{align}
Minimizing \eqref{obj:ssr} with respect to $\bm{\alpha}\coloneqq(\alpha_1,\dots,\alpha_{l+p})^\top$ results in over-fitting especially when knots are set at locations of $x_i$'s, i.e., $t_i=x_{i+1},i=0,\dots,l-1$ with $l=n$. 
To avoid over-fitting, it is common to minimize the objective penalized with %the smoothing
a certain regularization term:
\begin{align}
f(\bm{\alpha})&\coloneqq\frac{1}{2}\|\bm{y}-\bm{B}^{(p)}\bm{\alpha}\|_2^2+cR(\bm{\alpha}),
\label{eq:smoothing_spline_estimation}
\end{align}
where $R$ is a regularization function and $c\ge0$ is a constant for striking a balance between the SSR and the regularization term $R(\bm{\alpha})$. 
The most popular regularizer is the one defined by
\begin{align}
R(\bm{\alpha})=\frac{1}{2}\sum_{j=2}^{l+p-1}\big(\alpha_{j-1}-2\alpha_{j}+\alpha_{j+1}\big)^2.
\label{smoothing_regularizer}
\end{align}
The regularizer \eqref{smoothing_regularizer} can be considered as the quadratic penalty on the second difference of the coefficients of the adjacent basis functions since  $\alpha_{j-1}-2\alpha_{j}+\alpha_{j+1}=(\alpha_{j-1}-\alpha_{j})-(\alpha_{j}-\alpha_{j+1})$. 
The B-spline estimation using a certain order difference (possibly, not the second order one) as the regularizer is called P-spline \cite{eilers1996flexible}.
It is noteworthy that the P-spline regularizer with an arbitrary order difference is represented by
\begin{align}
R(\bm{\alpha})=\frac{1}{2}\|\bm{D}\bm{\alpha}\|_2^2,
\end{align}
where $\bm{D}$ is a difference matrix (see \cite{eilers1996flexible}).

\subsection{A motivating example and our strategy}
As reviewed in the previous section, the process for estimating a B-spline model depends on the basis functions or knots. 
Figure \ref{fig:two_estimation_results} shows two estimation results of P-spline with $p=3$ in Panels (a) and (b), each using a different set of knots of the same size.

\begin{figure}[ht]
    \centering
    \includegraphics[width=\columnwidth]{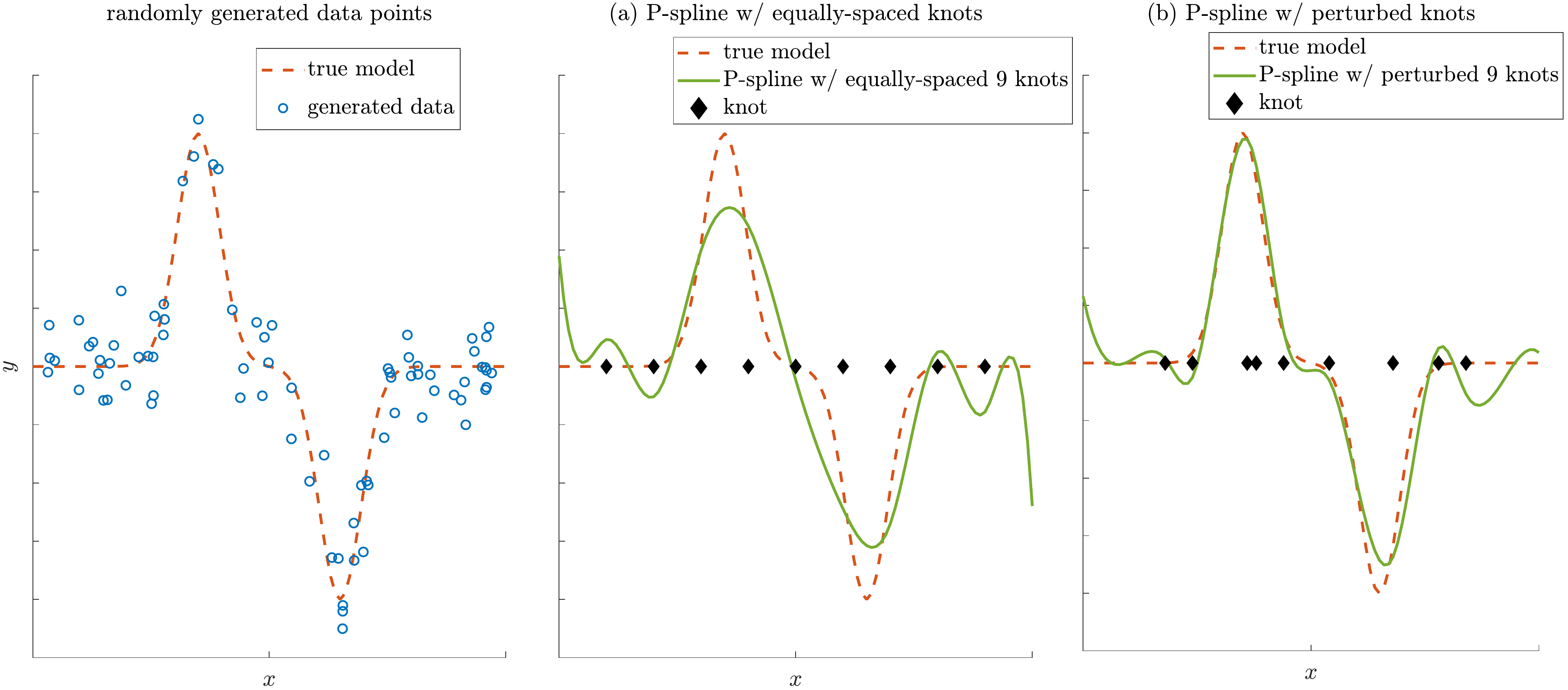}
    \caption{True data generating model and estimated models with two different sets of knots}
    \label{fig:two_estimation_results}
    \begin{quote}\footnotesize
        The leftmost picture depicts the true model $s(x)=0.8\exp(-[16(x - 0.35)]^2)-0.8\exp(-[16(x-0.65)]^2)$ (denoted by orange dashed line) and $80$ randomly data points (blue circles) whose $x_i$ were drawn from $\mathrm{U}(0,1)$ and $y_i$ were given by $s(x_i)+\epsilon_i$ where $\epsilon_i$ were drawn from $\mathrm{U}(0,0.1)$. The panels (a) and (b) show the estimated models (light green lines) computed over different sets of knots (black diamonds $\blacklozenge$). The model in Panel (a) is computed with the equally-spaced knots, and Panel (b) is computed over the knots whose locations are tweaked from the case of Panel (a) while keeping the number of knots the same. While the numbers of knots are the same, the fitting based on the knots with the perturbed locations is apparently improved. 
    \end{quote}
\end{figure}

From Figure \ref{fig:two_estimation_results} we see that the locations of knots affected the fitting, and it plays a significant role in improving the goodness of fit of the B-spline function. 

While it seems interesting to select the locations of knots for B-spline regression, it could be challenging since, 
as explained in the previous subsection, the (regularized) least square method for estimating a B-spine function follows computing the basis functions, which also follows giving a set of knots. 

For the purpose we consider the following strategy:
\begin{enumerate}
    \item Prepare a set of a large number $l-1$ of knots candidates.
    \item For a given integer $k(\ll l-1)$, estimate a model based on P-spline with the condition that allows the use of up to $k$ knots.   
\end{enumerate}
By providing a large number of knots candidates, the locations of the knots used to define the basis functions can be selected almost continuously.
However, it is still not obvious how to express the condition on the number of knots used. 
So in the next subsection we will explain how it is possible.

\subsection{Cardinality-constrained formulation}
To explain our formulation, let us start with the introduction of some notation.

Let 
\begin{align}
\bm{D}^{(1)}_+&\coloneqq
\begin{pmatrix}
0 & 1 &  & & \\
 & \ddots & \ddots & O & \\
 & O & \ddots & \ddots & \\
 & & & 0 & 1 \\
\end{pmatrix}
\in\mathbb{R}^{(l-1)\times l},\quad
\bm{D}^{(1)}_-\coloneqq
\begin{pmatrix}
1 & 0 &  & &  \\
 & \ddots & \ddots & O & \\
 & O & \ddots & \ddots & \\
 &  & & 1 & 0 \\
\end{pmatrix}
\in\mathbb{R}^{(l-1)\times l},
\end{align}
and 
\begin{align}
\bm{D}^{(1)}&\coloneqq\bm{D}^{(1)}_+-\bm{D}^{(1)}_-=
\begin{pmatrix}
-1 & 1 &  & & \\
 & \ddots & \ddots & O & \\
 & O & \ddots & \ddots & \\
 & & & -1 & 1 \\
\end{pmatrix}
\in\mathbb{R}^{(l-1)\times l}.
\end{align}
For $1\le q\le p$, define the $(q+1)$-th order counterparts by
\begin{align}
\bm{D}^{(q+1)}_+ & \coloneqq\bm{D}^{(q)}_+\bm{\Delta}^{(q+1)},\quad \bm{D}^{(q+1)}_-=\bm{D}^{(q)}_-\bm{\Delta}^{(q+1)},
\end{align}
and
\begin{align}
\bm{D}^{(q+1)} & \coloneqq\bm{D}^{(q)}\bm{\Delta}^{(q+1)} = (\bm{D}^{(q)}_+-\bm{D}^{(q)}_-)\bm{\Delta}^{(q+1)} = \bm{D}^{(q+1)}_+-\bm{D}^{(q+1)}_-,
\end{align}
where 
\begin{align}
\bm{\Delta}^{(q+1)} & \coloneqq
\begin{pmatrix}
\frac{-1}{t_1-t_{-q+1}} & \frac{1}{t_1-t_{-q+1}} &  &O& \\
 & \frac{-1}{t_2-t_{-q+2}} & \frac{1}{t_2-t_{-q+2}} & & \\
 & & \ddots & \ddots & \\
 &O& & \frac{-1}{t_{l-1+q}-t_{l-1}} & \frac{1}{t_{l-1+q}-t_{l-1}} \\
\end{pmatrix}
\in\mathbb{R}^{(l-1+q)\times (l+q)}
\label{eq:olDp}
\\
& = 
\begin{pmatrix}
\frac{1}{t_1-t_{-q+1}} &  &O& \\
 & \frac{1}{t_2-t_{-q+2}} & & \\
 & & \ddots & \\
 &O& &\frac{1}{t_{l-1+q}-t_{l-1}} \\
\end{pmatrix}
\begin{pmatrix}
-1 & 1 &  & & \\
 & \ddots & \ddots & O & \\
 & O & \ddots & \ddots & \\
 & & & -1 & 1 \\
\end{pmatrix}.
\label{eq:olDp_decomp}
\end{align}
Note that the rank of $\bm{D}^{(q+1)}$ remains to be $l-1$ because it is a scaled difference matrix.
The diagonal matrix in \eqref{eq:olDp_decomp} is for adjusting the spaces between adjacent knots.
If the equal-spaced knots are employed, the matrix $\bm{\Delta}^{(q+1)}$ turns out to be the ordinary difference matrix. 

\begin{theorem}\label{thm:active_knot_indicator}
Let $s=\sum_{j=-p}^{l-1}\alpha_{j+p+1}B^{(p)}_j$ and $\bm{\alpha}=(\alpha_1,\dots,\alpha_{p+l})^\top$. Then, $s$ does not use the 
$i$-th knot $t_{i}$ if and only if $(\bm{D}^{(p+1)}\bm{\alpha})_{i}=0$.
\end{theorem}

The proof is given in Appendix \ref{subsec:proof_active_knot_indicator}.
A simplified version of the proposition was displayed in the preprint of \cite{goepp2018spline}, but no proof was given therein.
In addition, the claim presented in \cite{goepp2018spline} was an incorrect assertion that Theorem \ref{thm:active_knot_indicator} holds for ordinary difference matrices without assuming the equally-spaced knots (or equivalently, the uniform intervals, namely, all $t_{i+1}-t_i$ are the same constant). On the other hand, Theorem \ref{thm:active_knot_indicator} applies to arbitrary locations of knots with the matrix $\bm{\Delta}^{(q+1)}$. 

Based on Theorem \ref{thm:active_knot_indicator}, the number of employed knots used in the spline can be gauged by the number of non-zero elements of the vector $\bm{D}^{(p+1)}\bm{\alpha}$. 
Motivated by this, we consider the modified formulation for B-spline regression as follows:
\begin{align}
\underset{\bm{\alpha}}{\mbox{minimize}} \quad & f(\bm{\alpha}) \label{obj:ssr0}\\
\mbox{subject to } & \|\bm{D}^{(p+1)}\bm{\alpha}\|_0 \leq K,
\label{eq:l0-constraint}
\end{align}
where $\|\bm{z}\|_0$ denotes the the number of non-zero elements of vector $\bm{z}$, and 
$K\in \{0,\dots,l-1\}$ is a user-defined (hyper-)parameter. 
With the interpretation of the term $\|\bm{D}^{(p+1)}\bm{\alpha}\|_0$ in mind, the interpretation of the hyper-parameter $K$ and the aim of the formulation \eqref{obj:ssr0}--\eqref{eq:l0-constraint} are clear. 

The formulation \eqref{obj:ssr0}--\eqref{eq:l0-constraint} includes two parameters, $c$ in the objective function and $K$ in the constraint, both of which limit the degrees of freedom of the spline and prevent overfitting in different ways. 
The user can calibrate them simultaneously in general, but in order to highlight the performance via the knot selection, we will focus mainly on $c=0$, and then the hyper-parameter becomes only $K$.

The formulation \eqref{obj:ssr0}--\eqref{eq:l0-constraint} is known to be intractable in that (i) the left-hand side of the constraint \eqref{eq:l0-constraint} is a discontinuous function; (ii) the constraint induces combinatorial complexity in the feasible set. 
To tame the intractability we use the continuous exact penalty representation developed in the context of sparse optimization \cite{gotoh2018dc}.

Let $z_{(j)}$ denote the $j$-th largest element of the vector $\bm{z}\coloneqq(z_1,\dots,z_d)^\top$, in absolute value, namely, $|z_{(1)}|\geq |z_{(2)}|\geq \cdots\geq |z_{(d)}|$. 
We define $T_K(\bm{z})$ by the sum of the smallest $d-K$ elements in absolute value, namely, 
\begin{align*}
    T_K(\bm{z})\coloneqq|z_{(K+1)}|+\cdots+|z_{(d)}|.
\end{align*}

It is easy to see that (i) $T_K$ is a continuous function, (ii) $T_K(\bm{z})\geq 0$ holds for any $\bm{z}\in\mathbb{R}^{d}$, and (iii) $T_K(\bm{z})=0$ if and only if $\|\bm{z}\|_0 \leq K$, (or, equivalently, $T_K(\bm{z})>0$ if and only if $\|\bm{z}\|_0 > K$). 
From (ii) and (iii), we see that the function $T_K(\bm{z})$ plays a role of the indicator of the fulfillment of the cardinality constraint $\|\bm{z}\|_0 \leq K$. 
This allows us to replace the constraint \eqref{eq:l0-constraint} with $T_K(\bm{D}^{(p+1)}\bm{\alpha})=0$, leading to an equivalent formulation to \eqref{obj:ssr0}--\eqref{eq:l0-constraint}:
\begin{align}
\underset{\bm{\alpha}}{\mbox{minimize}} \quad & f(\bm{\alpha}) \label{obj:ssr2}\\
\mbox{subject to } & T_K(\bm{D}^{(p+1)}\bm{\alpha})=0.
\label{eq:trimmed-regularized-constraint}
\end{align}
This new formulation \eqref{obj:ssr2}--\eqref{eq:trimmed-regularized-constraint} looks to mitigate the intractability associated with \eqref{obj:ssr0}--\eqref{eq:l0-constraint}, in that the new constraint \eqref{eq:trimmed-regularized-constraint} does not involve any discontinuity unlike \eqref{eq:l0-constraint}. On the other hand, it is still not so tractable in that it keeps having a constraint. 

Associated with the constrained-form optimization problem \eqref{obj:ssr2}--\eqref{eq:trimmed-regularized-constraint}, let us consider the penalty form 
\begin{align}
\underset{\bm{\alpha}}{\mbox{minimize}} \quad & f(\bm{\alpha}) + \gamma T_K(\bm{D}^{(p+1)}\bm{\alpha}), \label{obj:trimmed_penalized}
\end{align}
where $\gamma>0$ is a parameter. 
The function $T_K(\bm{z})$ is sometimes referred to as the {\it trimmed $\ell_1$ regularizer} or {\it partial $\ell_1$ regularizer}, since it can be decomposed as
\begin{align}
T_K(\bm{z})=\|\bm{z}\|_1-\underbrace{(|z_{(1)}|+|z_{(2)}|+\cdots+|z_{(K)}|)}_{\mbox{largest-}K\mbox{ norm}~\equiv~\mbox{trimmed part}}.
\end{align}
The formulation \eqref{obj:trimmed_penalized} is an example of the generalized trimmed lasso \citep{yagishita2022exact}.
Note that the local optimality and d-stationarity of \eqref{obj:trimmed_penalized} are equivalent \citep[Proposition 2]{yagishita2022exact}, where we call $\bm{z}^*$ a d(irectional)-stationary point of $\min_{\bm{z}}\phi(z)$ if $\phi'(\bm{z}^*;\bm{d})\coloneqq\lim_{\xi\searrow 0}\frac{\phi(\bm{z}^*+\xi\bm{d})-\phi(\bm{z}^*)}{\xi}\ge0$ holds for all $\bm{d}$.

The formulation \eqref{obj:trimmed_penalized} looks easier to approach in that we can adapt more algorithms for unconstrained optimization. On the other hand, a new hyper-partameter $\gamma$ is introduced in \eqref{obj:trimmed_penalized}, which might blur the interpretation of the constrained formulation \eqref{obj:ssr2}--\eqref{eq:trimmed-regularized-constraint}.
However, as will be presented later in Theorem \ref{thm:e.p.p.}, we can show that for a sufficiently large $\gamma$, any d-stationary point of the penalty-form \eqref{obj:trimmed_penalized} satisfies the constraint \eqref{eq:trimmed-regularized-constraint} when $c=0$ \citep[Theorem 3]{yagishita2022exact} (the case $c>0$ can be proved using the same approach as for Theorem \ref{thm:e.p.p.}, which is described later).
Therefore, we do not have to bother selecting the value of $\gamma$ from the viewpoint of the interpretation, and accordingly, in this sense, \eqref{obj:trimmed_penalized} only has the single and more intuitive hyper-parameter $K$. (We will recall this after Theorem \ref{thm:e.p.p.}.)

\section{Solution Method}\label{sec:solution-method}
In this section we present an approach that enables us to attain a d-stationary point of \eqref{obj:trimmed_penalized} by applying a proximal gradient-type algorithm after some variable change. 

One of the nice features of the trimmed function $T_K$ is that its proximal operation is easy to implement. 
For a function $\phi:\mathbb{R}^d\to\mathbb{R}\cup\{\infty\}$, the proximal mapping is defined by
\begin{align}
\prox_\phi(\bm{a})\coloneqq\argmin_{\bm{z}}
\big\{ \phi(\bm{z})+\frac{1}{2}\|\bm{z}-\bm{a}\|_2^2\big\}.
\end{align}
It is known that if $\phi$ is a proper closed convex function, the mapping is a singleton.
The trimmed function $T_K$ is not convex, and its proximal mapping is not necessarily to be a singleton. Furthermore, the minimization in the definition of the mapping is a nonconvex optimization.

However, it is known that for a point $\bm{a}=(a_1,\dots,a_d)^\top$, the vector $\bm{a}^+$ defined by
\begin{align}
    (\bm{a}^+)_i=
    \begin{cases}
        a_j, & j\in\mathcal{J},\\
        \mathrm{soft}_{\lambda}(a_j),&j\notin\mathcal{J},
    \end{cases}
\end{align}
is a point in $\prox_{\lambda T_K}(\bm{a})$, where $\mathcal{J}$ is the index set of largest $K$ components of the absolute values of $\bm{a}$, and 
\begin{align}
    \mathrm{soft}_{\lambda}(a)\coloneqq
    \begin{cases}
        a+\lambda,&a<-\lambda,\\
        0,             &-\lambda\le a\le\lambda,\\
        a-\lambda,&a>\lambda,
    \end{cases}
\end{align}
is the soft-thresholding operator on $\mathbb{R}$ (see Theorem 5.4 of \citet{lu2018sparse} or Subsection C.3 of \citet{bertsimas2017trimmed} for the derivation).
Therefore, we can easily find a point in the set $\prox_{T_K}(\bm{a})$ just by finding out largest $K$ elements of the vector $\bm{a}$ and applying the formula given above.

Although the availability of the proximal mapping of $T_K$ is a good news since it might enable us to apply some kinds of proximal gradient-type algorithms. However, unfortunately, we cannot directly apply the proximal gradient algorithms since the related term in the problem \eqref{obj:trimmed_penalized} comes up with the difference matrix $\bm{D}^{(p+1)}$. 
Namely, the proximal map of $T_K(\bm{D}^{(p+1)}\cdot)$ may not be tractable.

Another option is to use the alternating direction method of multipliers (ADMM) (see, e.g., \citep{boyd2011distributed}). 
The global convergence of the ADMM to d-stationary point has been shown \citep[Theorem 3]{yagishita2020pursuit}.
However, under parameter setting where global convergence is guaranteed, practical convergence of the ADMM for \eqref{obj:trimmed_penalized} is slow, as can be seen from Figure \ref{fig:vs-admm} in Subsection \ref{subsec:vs-admm}.

To overcome the difficulty, we here apply a change of variables, following \citet{kim2009ell_1,tibshirani2011solution} and \citet{tibshirani2014adaptive}. 
More specifically, in order to get the representation of the form $T_{K}(\bm{\beta})$ for a vector $\bm{\beta}$ in place of $T_{K}(\bm{D}^{(p+1)}\bm{\alpha})$, we consider a change of variables
\begin{align}
    \bm{\beta}=\bm{D}^{(p+1)}\bm{\alpha}.
\end{align}
We can construct a square matrix
\begin{align}
    \hat{\bm{D}}^{(p+1)}\coloneqq
    \begin{pmatrix}
        \bm{D}^{(p+1)}\\
        \bm{A}
    \end{pmatrix}
    \in\mathbb{R}^{(l+p)\times(l+p)}
\end{align}
so that it will be non-singular using an adequate matrix $\bm{A}$ since $\bm{D}^{(p+1)}$ is of rank $l-1$.
With such an invertible matrix $\hat{\bm{D}}^{(p+1)}$, consider the one-to-one variable change
\begin{align}
    \begin{pmatrix}
        \bm{\beta}\\
        \bm{\beta}'
    \end{pmatrix}
    =\hat{\bm{D}}^{(p+1)}\bm{\alpha}
    \quad\leftrightarrow\quad \bm{\alpha}=\bm{\Sigma}^{(p+1)}
    \begin{pmatrix}
        \bm{\beta}\\
        \bm{\beta}'
    \end{pmatrix},
\end{align}
where $\bm{\Sigma}^{(p+1)}\coloneqq\big(\hat{\bm{D}}^{(p+1)}\big)^{-1}$.
An example of $\hat{\bm{D}}^{(p+1)}$ (or $\bm{A}$) and the fact that $\bm{\Sigma}^{(p+1)}$ can be easily computed for the example are shown in Appendix \ref{sec:diff-mat}. 
Using this, the problem \eqref{obj:trimmed_penalized} can be equivalently rewritten as
\begin{align}
\underset{\bm{\beta},\bm{\beta}'}{\mbox{minimize}} \quad &
G(\bm{\beta},\bm{\beta}')\coloneqq\frac{1}{2}\big\|\bm{y}-\bm{B}^{(p)}\bm{\Sigma}_1\bm{\beta}-\bm{B}^{(p)}\bm{\Sigma}_2\bm{\beta}'\big\|_2^2+\frac{c}{2}\big\|\bm{D}\bm{\Sigma}_1\bm{\beta}+\bm{D}\bm{\Sigma}_2\bm{\beta}'\big\|_2^2+\gamma T_K(\bm{\beta}), \label{obj:trimmed_penalized_TF}
\end{align}
where $\bm{\Sigma}_1$ and $\bm{\Sigma}_2$ are $(l+p)\times(l-1)$ and $(l+p)\times(p+1)$ (sub)matrices, respectively, such that
\begin{align}
    \bm{\Sigma}^{(p+1)}=
    \begin{pmatrix}
        \bm{\Sigma}_1 & \bm{\Sigma}_2
    \end{pmatrix}.
\end{align}

Although proximal gradient algorithms can be applied to \eqref{obj:trimmed_penalized_TF} efficiently, we further consider reducing variables by minimizing \eqref{obj:trimmed_penalized_TF} with respect to $\bm{\beta}'$.
We additionally consider the following assumption. 
\begin{assumption}\label{assump:pdf}
$\bm{B}^{(p)}\bm{\Sigma}_2$ is a full column rank matrix. 
\end{assumption}
Note that Assumption \ref{assump:pdf} implies that the Hessian of $G(\bm{\beta},\bm{\beta}')$ with respect to $\bm{\beta}'$ is positive definite, and $\arg\min_{\bm{\beta}'}G(\bm{\beta},\bm{\beta}')$ is a singleton. 
It is easy to see that Assumption \ref{assump:pdf} holds if and only if the fitting of (non-piecewise) polynomial of order $p$ to the data by the ordinary least squares is uniquely determined, and hence it is not restrictive to suppose it. 

Under Assumption \ref{assump:pdf}, the problem \eqref{obj:trimmed_penalized_TF} is reduced to
\begin{align}
\underset{\bm{\beta}}{\mbox{minimize}} \quad
F(\bm{\beta})\coloneqq\frac{1}{2}\big\|\bm{z}_1-\bm{L}_1\bm{\beta}\big\|_2^2+\frac{c}{2}\big\|\bm{z}_2-\bm{L}_2\bm{\beta}\big\|_2^2+\gamma T_K(\bm{\beta}), \label{obj:trimmed_penalized_TF2}
\end{align}
where
\begin{align}
    \bm{H}_1 &\coloneqq\big((\bm{B}^{(p)}\bm{\Sigma}_2)^\top\bm{B}^{(p)}\bm{\Sigma}_2+c(\bm{D}\bm{\Sigma}_2)^\top\bm{D}\bm{\Sigma}_2\big)^{-1}(\bm{B}^{(p)}\bm{\Sigma}_2)^\top,\\
    \bm{H}_2 &\coloneqq\big((\bm{B}^{(p)}\bm{\Sigma}_2)^\top\bm{B}^{(p)}\bm{\Sigma}_2+c(\bm{D}\bm{\Sigma}_2)^\top\bm{D}\bm{\Sigma}_2\big)^{-1}\big\{(\bm{B}^{(p)}\bm{\Sigma}_2)^\top\bm{B}^{(p)}\bm{\Sigma}_1+c(\bm{D}\bm{\Sigma}_2)^\top\bm{D}\bm{\Sigma}_1\big\},\\
    \bm{z}_1 &\coloneqq\bm{y}-\bm{B}^{(p)}\bm{\Sigma}_2\bm{H}_1\bm{y}, \quad \bm{z}_2\coloneqq\bm{D}\bm{\Sigma}_2\bm{H}_1\bm{y}, \quad \bm{L}_1\coloneqq\bm{B}^{(p)}\bm{\Sigma}_1-\bm{B}^{(p)}\bm{\Sigma}_2\bm{H}_2, \quad \bm{L}_2\coloneqq\bm{D}\bm{\Sigma}_2\bm{H}_2-\bm{D}\bm{\Sigma}_1.
\end{align}
The local optimality and d-stationarity of \eqref{obj:trimmed_penalized_TF2} are equivalent for the same reason as \eqref{obj:trimmed_penalized}.
The relationship between solutions of \eqref{obj:trimmed_penalized} and \eqref{obj:trimmed_penalized_TF2} is as follows.

\begin{proposition}\label{prop:relation}
Suppose Assumption \ref{assump:pdf}, and 
let $g(\bm{\beta})\coloneqq\argmin_{\bm{\beta}'}G(\bm{\beta},\bm{\beta}')=\bm{H}_1\bm{y}-\bm{H}_2\bm{\beta}$.
\begin{enumerate}[(a)]
    \item If $\bm{\beta}^*$ be a local minimum of \eqref{obj:trimmed_penalized_TF2}, then $\bm{\Sigma}^{(p+1)}(\bm{\beta}^{*\top},g(\bm{\beta}^*)^\top)^\top$ is also locally optimal to \eqref{obj:trimmed_penalized}.
    \item If $\bm{\beta}^*$ be an optimal solution of \eqref{obj:trimmed_penalized_TF2}, then $\bm{\Sigma}^{(p+1)}(\bm{\beta}^{*\top},g(\bm{\beta}^*)^\top)^\top$ is also optimal to \eqref{obj:trimmed_penalized}.
\end{enumerate}
\end{proposition}
The proof of Proposition \ref{prop:relation} is given in Appendix \ref{subsec:proof_relation}.

The statement (a) of Proposition \ref{prop:relation} will be more useful than (b) later, in that global optimality of nonconvex optimization is in general hard to guarantee within a practical amount of time for large-sized problems and we can prove that the employed proximal-gradient algorithm (Algorithm \ref{alg:GIST}) generates sequences that converge to local minima of \eqref{obj:trimmed_penalized_TF2} and \eqref{obj:trimmed_penalized} (Theorem \ref{thm:global-convergence}). 

Using Lemma 4 of \citet{yagishita2022exact} and Proposition \ref{prop:relation}, we have the following result, which is an extension of Theorem 2.1 of \citet{amir2021trimmed}.
See Appendix \ref{subsec:proof_e.p.p.} for the proof.
\begin{theorem}\label{thm:e.p.p.}
Suppose Assumption \ref{assump:pdf}, and that
\begin{align}\label{eq:e.p.p.}
    \gamma>\max_{j=1,\ldots,l-1}\left\{\|\bm{l}^{(1)}_j\|_2+\sqrt{c}\|\bm{l}^{(2)}_j\|_2\right\}\sqrt{\|\bm{z}_1\|_2^2+c\|\bm{z}_2\|_2^2},
\end{align}
where $\bm{l}^{(1)}_j$ and $\bm{l}^{(2)}_j$ denote the $j$-th column of $\bm{L}_1$ and $\bm{L}_2$, respectively.
Let $\bm{\beta}^*$ be a local minimum of \eqref{obj:trimmed_penalized_TF2}.
Then it holds that $\|\bm{\beta}^*\|_0\le K$, which implies that $\bm{\Sigma}^{(p+1)}(\bm{\beta}^{*\top},g(\bm{\beta}^*)^\top)^\top$ is a local minimum of \eqref{obj:trimmed_penalized} satisfying the constraint \eqref{eq:l0-constraint}.
\end{theorem}

Theorem \ref{thm:e.p.p.} indicates that we can obtain a locally optimal solution of \eqref{obj:ssr0}--\eqref{eq:l0-constraint} by solving the problem \eqref{obj:trimmed_penalized_TF2} for sufficiently large $\gamma$. 
In this sense, the penalty parameter $\gamma$ no longer needs tuning, only the more interpretable parameter $K$.

To obtain a locally optimal solution of \eqref{obj:trimmed_penalized_TF2}, we propose using the general iterative shrinkage and thresholding (GIST) algorithm \citep{wright2009sparse,gong2013general}.
In addition to the proximal operation at each iteration, the GIST employs a line search, which is enhanced with the so-called Barzilai-Borwein (BB) initialization rule \citep{barzilai1988two} and a nonmonotone line search \citep{grippo1986nonmonotone,grippo2002nonmonotone}. 
With the BB rule \eqref{eq:BB}, we can set the (initial) step size $\eta_t$ so that the line search can start from the stepsize which reflects the curvature of the smooth term of the objective function. 
On the other hand, with the termination condition \eqref{eq:nmt} for the line search, we can quit the line search while approximately ensuring the decrease of the objective values. 
The algorithm we apply is described in Algorithm \ref{alg:GIST}, where we denote by $h$ the first two terms of \eqref{obj:trimmed_penalized_TF2}, namely, $h(\bm{\beta})\coloneqq\frac{1}{2}\big\|\bm{z}_1-\bm{L}_1\bm{\beta}\big\|_2^2+\frac{c}{2}\big\|\bm{z}_2-\bm{L}_2\bm{\beta}\big\|_2^2$.

\begin{algorithm}[tb]
   \caption{General Iterative Shrinkage and Thresholding Algorithm for \eqref{obj:trimmed_penalized_TF2}}
   \label{alg:GIST}
\begin{algorithmic}
\STATE {\bf Input:} Problem \eqref{obj:trimmed_penalized_TF}; $\rho>1$; $0<\underline{\eta}\le\overline{\eta}$; $\sigma\in(0,1)$; $M\geq 1$ an integer.
\STATE {\bf Initialize:} $\bm{\beta}_0\in\mathbb{R}^d$, $\eta_0\in\big[\underline{\eta}/\rho,\overline{\eta}/\rho\big]$, $t\leftarrow0$
\REPEAT
 \REPEAT
 \STATE $\eta_t\leftarrow \rho\eta_t$
 \STATE $\bm{\beta}_{t+1}\in \mathrm{Prox}_{\gamma T_K/\eta_t}\left(\bm{\beta}_t -\frac{1}{\eta_t}\nabla h(\bm{\beta}_t)\right)$
 \UNTIL{line search condition 
 \begin{align}
F(\bm{\beta}_{t+1})\le\max_{\max\{t-M+1,0\}\le s\le t} F(x^s)-\frac{\sigma\eta_t}{2}\|\bm{\beta}_{t+1}-\bm{\beta}_t\|^2,
\label{eq:nmt}
\end{align}
holds.
 }
\STATE Set
\begin{align}
\eta_{t+1}
=
\min\left\{\overline{\eta}, \max\left\{\underline{\eta},\frac{(\nabla h(\bm{\beta}_{t+1})-\nabla h(\bm{\beta}_{t}))^\top(\bm{\beta}_{t+1}-\bm{\beta}_{t})}{\|\bm{\beta}_{t+1}-\bm{\beta}_{t}\|^2}\right\}\right\}/\rho,
\label{eq:BB}
\end{align}
\STATE Set $t\leftarrow t+1$.
\UNTIL{some stopping criterion is satisfied}
\end{algorithmic}
\end{algorithm}

For the application of Algorithm \ref{alg:GIST} to \eqref{obj:trimmed_penalized_TF2}, we have the following global convergence result.
The proof is given in Appendix \ref{subsec:proof_global-convergence}.

\begin{theorem}\label{thm:global-convergence}
Any accumulation point of $\{\bm{\beta}_t\}$ is a local minimum of \eqref{obj:trimmed_penalized_TF2}, and any accumulation point of $\big\{\bm{\Sigma}^{(p+1)}(\bm{\beta}_t^\top,g(\bm{\beta}_t)^\top)^\top\big\}$ is a local minimum of \eqref{obj:trimmed_penalized}.
\end{theorem}

\citet{nakayama2021superiority} have shown the global convergence of the GIST to d-stationary points for a general composite optimization problem.
Although applying Theorem 1 of \citet{nakayama2021superiority} to our formulation \eqref{obj:trimmed_penalized_TF2} requires the boundedness of the generated sequence of the GIST to be assumed, Theorem \ref{thm:global-convergence} does not require it.

\section{Numerical Examples}\label{sec:experiments}
To demonstrate how the proposed methodology works, some numerical examples are presented in this section.
The first two subsections are devoted to demonstrating the validity of our solution method.
In the last subsection, effectiveness of our methodology is shown by comparing with the A-spline \citep{goepp2018spline} that aspires to select knots automatically based on Theorem \ref{thm:active_knot_indicator} just like ours.

Through the experiments, we consider the case $p=3$, namely, cubic spline and equally-spaced knots. 
To focus on the regularization by limiting the number of knots, we did not to apply the smoothing regularization \eqref{smoothing_regularizer}, i.e., we set $c=0$, except to compare the sensitivity of the computation time to the roughness penalty parameter $c$ in Figure \ref{fig:vs-monotone-c}, which is based on the ordinary second order difference matrix as $\bm{D}$ in \eqref{smoothing_regularizer}. 
Based on Theorem \ref{thm:e.p.p.}, the parameter $\gamma$ of \eqref{obj:trimmed_penalized}, \eqref{obj:trimmed_penalized_TF}, or \eqref{obj:trimmed_penalized_TF2} was set as $1.001$ times the right-hand side of \eqref{eq:e.p.p.} so as to ensure $\gamma$ is large enough to serve as the exact penalty. 
As for setting of Algorithm \ref{alg:GIST}, we used $\rho=2,~ \underline{\eta}=10^{-6} ,~ \overline{\eta}=10^6,~ \sigma=10^{-2},~ \bm{\beta}_0=0$, and $\eta_0=1$.
All the real data sets used here are available in the \texttt{R} package \texttt{SemiPar} and response variables of them were standardized.
We set $t_0=0,~ t_l=1$ for synthetic data sets and $t_0=\min_ix_i-10^{-3}\times(\max_ix_i-\min_ix_i),~ t_l=\max_ix_i+10^{-3}\times(\max_ix_i-\min_ix_i)$ for real data sets.

\subsection{Comparison with ADMM and GIST without reducing variables}\label{subsec:vs-admm}
To see the effect of the change of variables and reducing variables, our solution method is compared with the ADMM for \eqref{obj:trimmed_penalized} and the GIST for \eqref{obj:trimmed_penalized_TF}.
We consider applying them with $c=0$ to the LIDAR data ($n=221$) and the term structure data ($n=117$).
Each algorithm solved the problems of selecting $K=10$ knots from $l=50, 100$ candidates.

The penalty parameter of ADMM was determined to have global convergence property (see Subsection 4.4 of \citet{yagishita2020pursuit}).
The parameter settings for the GIST for \eqref{obj:trimmed_penalized_TF} were the same as for Algorithm \ref{alg:GIST}, and the monotone line search ($M=1$) and nonmonotone line search ($M=10$) were considered for them.
Initial points of the ADMM and the GIST were set as $\bm{\Sigma}^{(p+1)}(0^\top,(\bm{H}_1\bm{y})^\top)^\top$ and $(0^\top,(\bm{H}_1\bm{y})^\top)^\top$, and hence initial function values of them and our solution method have the same value.

\begin{figure}[ht]
    \centering
    \includegraphics[width=\columnwidth]{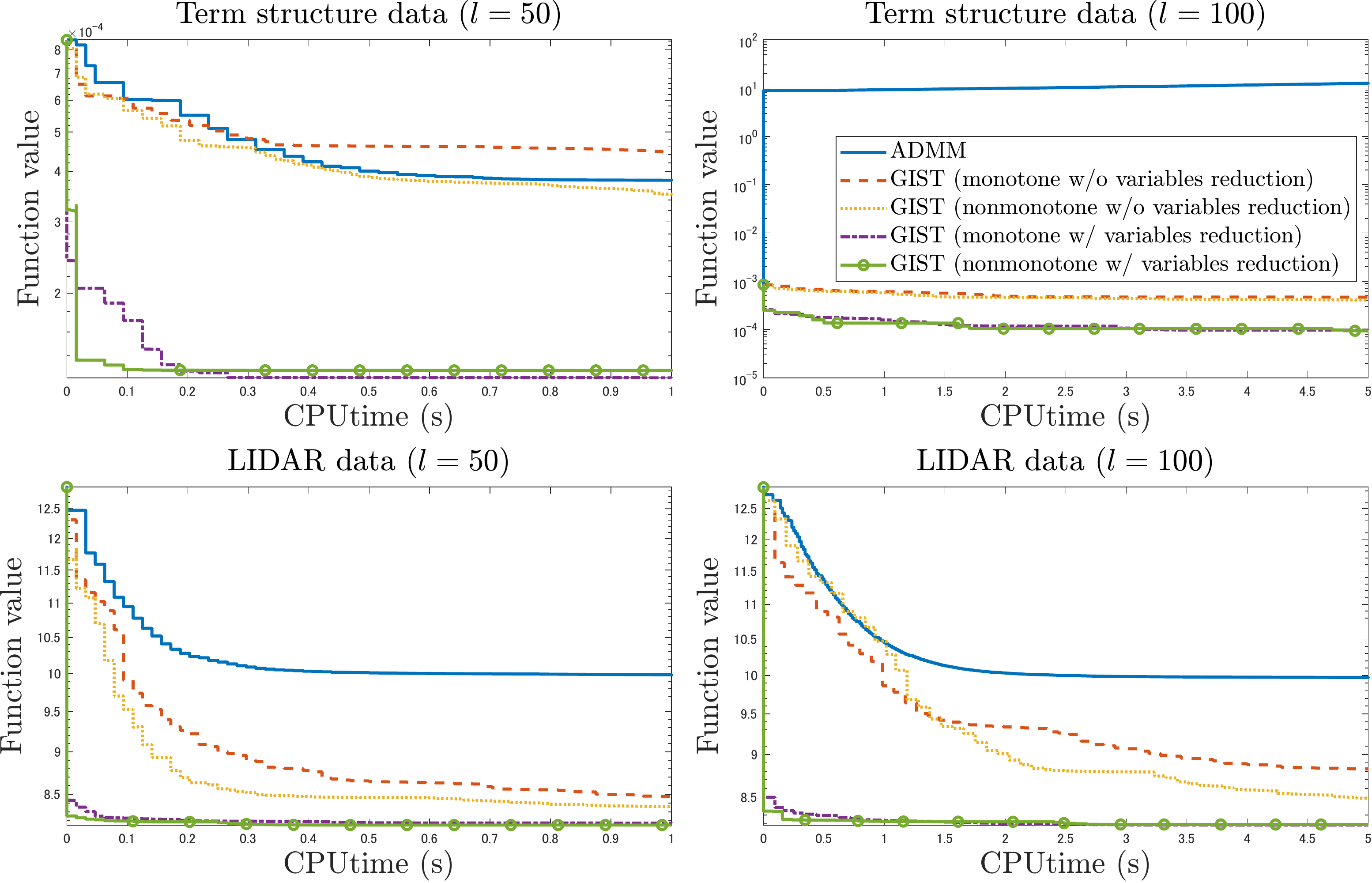}
    \caption{Convergence behaviour of each method for the LIDAR data and term structure data}
    \label{fig:vs-admm}
\end{figure}

Figure \ref{fig:vs-admm} shows plots of the objective function value per CPU time in seconds.
We see from Figure \ref{fig:vs-admm} that our proposed methods converge faster than the others.
Especially, the ADMM behaves very poorly for the term structure data with $l=100$ where the number of data samples, $n$, is not large enough for the number of candidates of knots, $l$.
This implies that the ADMM is not appropriate for our aim of selecting knots from many knots candidates.
On the other hand, although the nonmonotone line search seems to be a bit faster than the monotonic one, it is not clear.
So we investigate the effectiveness of the nonmonotone line search in detail in the next subsection.

\subsection{Effectiveness of nonmonotone line search}\label{subsec:vs-monotone}
To investigate the effectiveness of the nonmonotone line search ($M=10$), time to convergence was compared to the monotonic one ($M=1$) under several settings.
More precisely, we measured CPU times to fulfill the stopping criterion $\|\bm{\beta}_t-\bm{\beta}_{t-1}\|_2\le\sqrt{K(l-1)n}\times10^{-6}$ for various $l,~ K$, and $c$.

The synthetic data %to be 
used was generated as
\begin{align}
    y_i=\sum_{j=-p}^{l-1}\Tilde{\alpha}_{j+p+1}B^{(p)}_j(x_i)+\epsilon_i, \quad i=1,\ldots,200,
\end{align}
where $x_i,~ \epsilon_i$, and $\Tilde{\alpha}_j$ were independently drawn from $\mathrm{U}(0,1)$, $\mathrm{N}(0,0.1^2)$, and $\mathrm{N}(0,1)$, respectively.

Figures \ref{fig:vs-monotone-l}, \ref{fig:vs-monotone-K}, and \ref{fig:vs-monotone-c} show box plots of $\log_2(\tau_\mathrm{mono}/\tau_\mathrm{non})$ over 100 repetitions, where $\tau_\mathrm{mono}$ and $\tau_\mathrm{non}$ denote the computation times of the GIST with the monotone line search and the nonmonotone one, respectively.
For any parameter setting, we see from Figures \ref{fig:vs-monotone-l}, \ref{fig:vs-monotone-K}, and \ref{fig:vs-monotone-c} that the GIST with nonmonotone line search converged faster than the monotone one in 75 percent of the instances.
In addition, the computation time of the nonmonotone one was less than half that of the monotone one in 50 percent of the instances for almost all parameter settings.
Based on the result, we recommend the use of the nonmonotone line search and applied it in the following experiment.

\begin{figure}[p]
    \centering
    \includegraphics[width=\columnwidth]{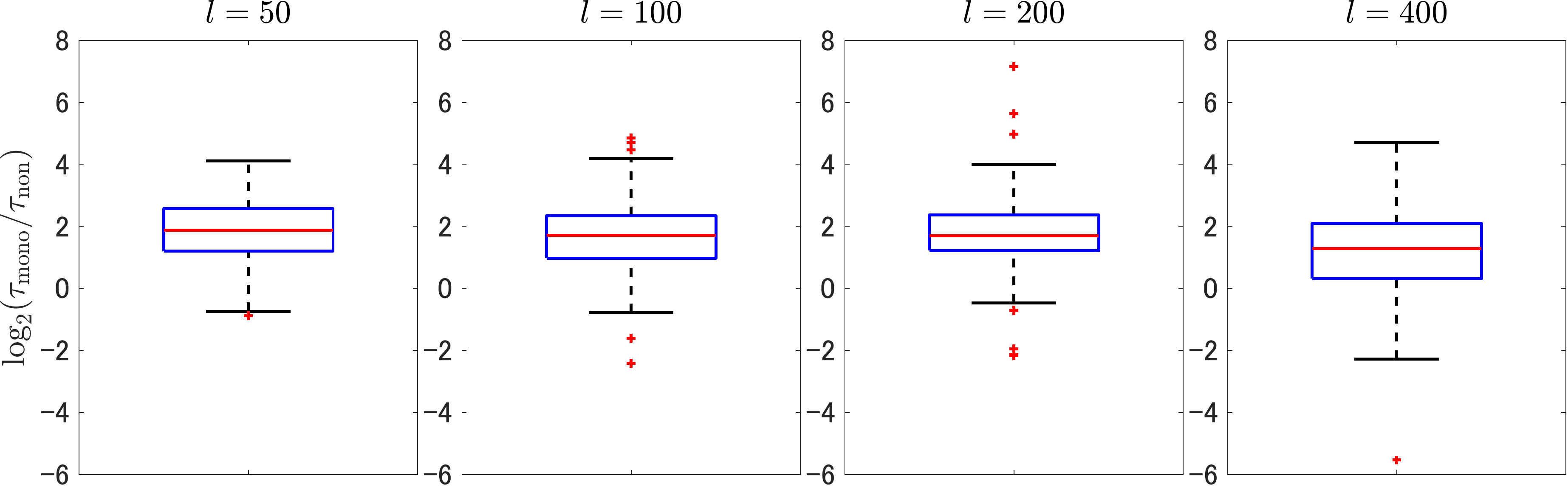}
    \caption{Box plots of $\log_2(\tau_\mathrm{mono}/\tau_\mathrm{non})$ for several $l$ with $K=10$ and $c=0$}
    \label{fig:vs-monotone-l}
    \vspace{20pt}
    \includegraphics[width=\columnwidth]{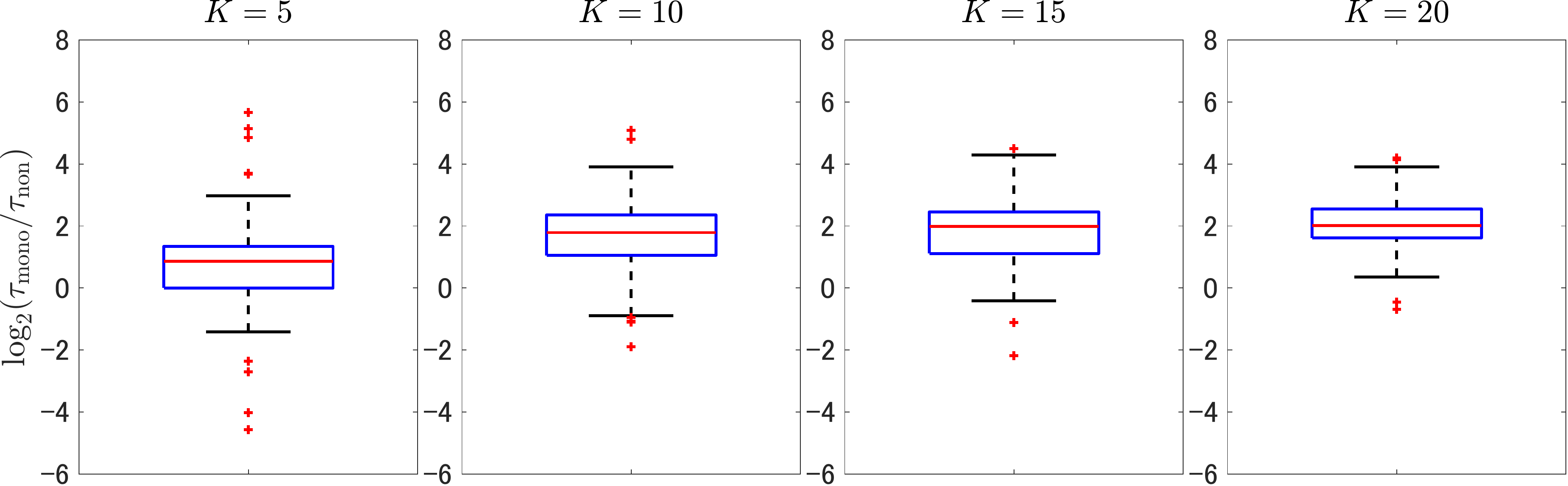}
    \caption{Box plots of $\log_2(\tau_\mathrm{mono}/\tau_\mathrm{non})$ for several $K$ with $l=100$ and $c=0$}
    \label{fig:vs-monotone-K}
    \vspace{20pt}
    \includegraphics[width=\columnwidth]{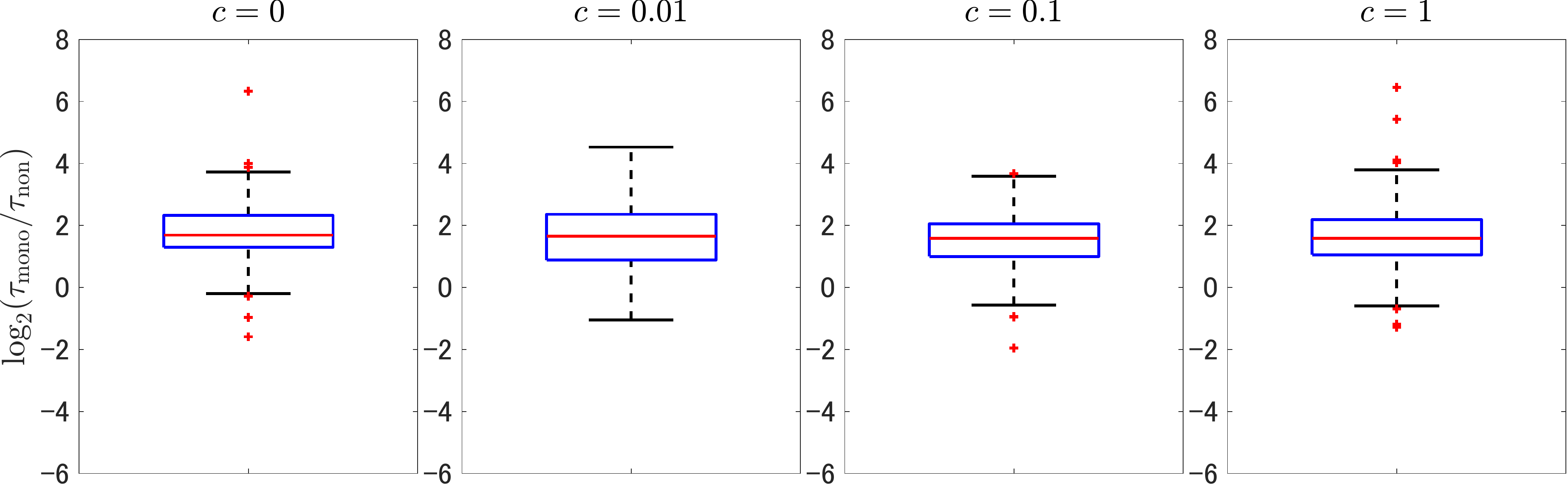}
    \caption{Box plots of $\log_2(\tau_\mathrm{mono}/\tau_\mathrm{non})$ for several $c$ with $K=10$ and $l=100$}
    \label{fig:vs-monotone-c}
\end{figure}

\subsection{Comparison with A-spline}\label{subsec:vs-a-spline}
In this subsection, our proposed methodology is compared with the A-spline \citep{goepp2018spline}.
The A-spline is a heuristics that tries to minimize the discontinuous objective function:
\begin{align}
    \underset{\bm{\alpha}}{\mbox{minimize}} \quad & \frac{1}{2}\|\bm{y}-\bm{B}^{(p)}\bm{\alpha}\|_2^2+\frac{\lambda}{2}\|\bm{D}^{(p+1)}\bm{\alpha}\|_0
\end{align}
by the adaptive ridge \citep{rippe2012visualization,frommlet2016adaptive}, which repeats
\begin{align}\label{eq:adaptive-procedure}
    \bm{\alpha}^+\leftarrow\left({\bm{B}^{(p)}}^\top\bm{B}^{(p)}+\lambda{\bm{D}^{(p+1)}}^\top\bm{W(\alpha)}\bm{D}^{(p+1)}\right)^{-1}{\bm{B}^{(p)}}^\top\bm{y},
\end{align}
where $\lambda>0$ is a trade-off parameter to strike the balance between the SSR and the cardinality term and $\bm{W(\alpha)}$ is a weight matrix determined by $\bm{\alpha}$.
After iterations, the re-estimation based on \eqref{obj:ssr} using only the knots selected in the obtained solution is performed.

In the experiment, $\lambda$ was chosen to minimize the Bayesian information criterion (BIC) among the hundred values,  $1^2\times10^{-4},2^2\times10^{-4},\ldots,100^2\times10^{-4}$, since the choice of $\lambda$ with the BIC has been reported to be suitable for the A-spline \citep{goepp2018spline}. 
As for the settings of the GIST, the nonmonotone line search ($M=10$) and $c=0$ were used.
The GIST was terminated either when $\|\bm{\beta}_t-\bm{\beta}_{t-1}\|_2\le\sqrt{K(l-1)n}\times10^{-6}$, or when the number of iterations reached $10^5$.
After termination, we executed the same re-estimation as the A-spline. 
The parameter $K$ was determined based on the BIC minimization among the twenty values, $1,2,\ldots,20$. 

\begin{figure}[t]
    \centering
    \includegraphics[width=\columnwidth]{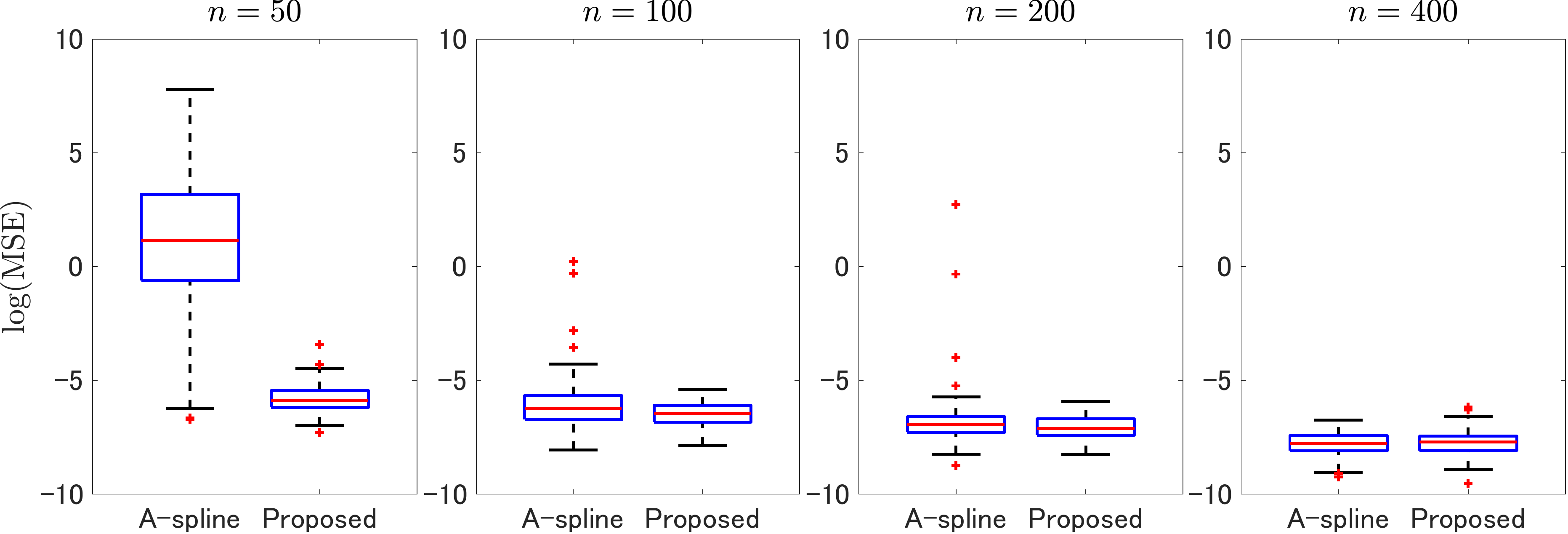}
    \caption{Box plots of the logarithm of the MSE for several $n$ with $l=100$}
    \label{fig:vs-a-spline-1}
    \vspace{20pt}
    \includegraphics[width=\columnwidth]{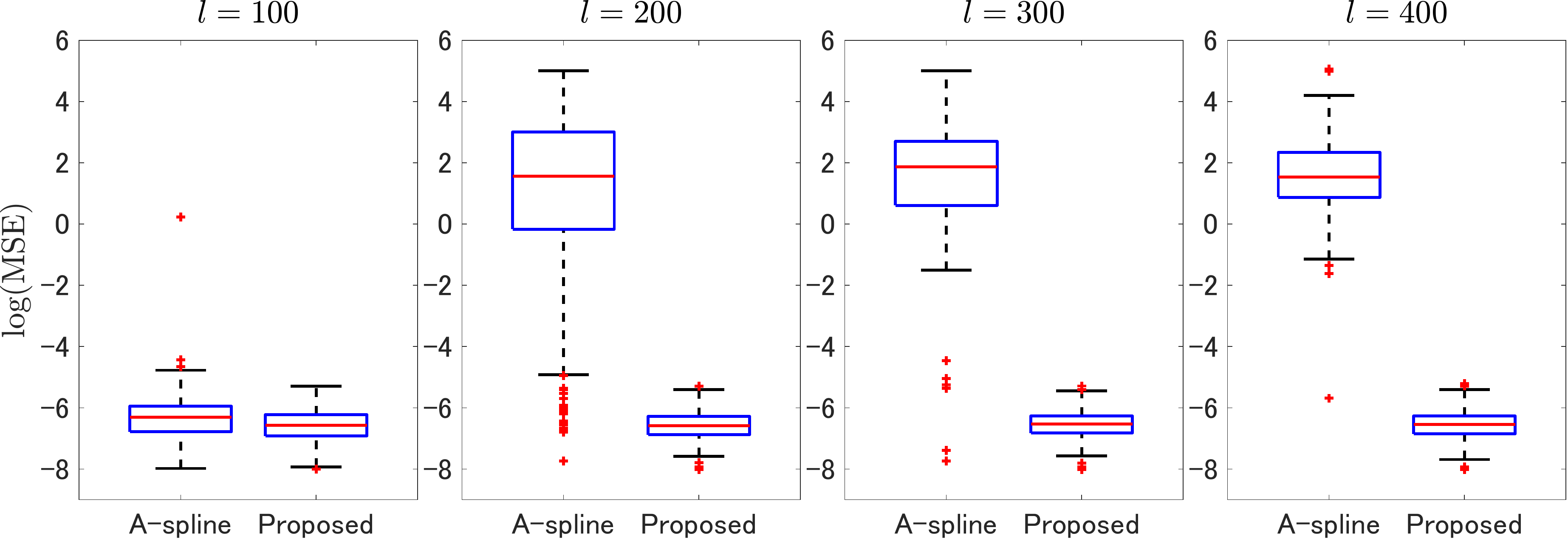}
    \caption{Box plots of the logarithm of the MSE for several $l$ with $n=100$}
    \label{fig:vs-a-spline-2}
\end{figure}

We first consider synthetic data sets under several $n$ and $l$.
As can be seen from \eqref{eq:adaptive-procedure}, the A-spline can be numerically unstable depending on the data (especially, when $n<l+p$), so when this happens, the data set was discarded and a new one was generated.
Each data set was generated randomly as
\begin{align}
    y_i=\sum_{j=-p}^{5}\Tilde{\alpha}_{j+p+1}\Tilde{B}^{(p)}_j(x_i)+\epsilon_i, \quad i=1,\ldots,n,
\end{align}
where $\Tilde{B}^{(p)}_j$ is the B-spline basis of order $p$ under $t_{-p}<\cdots<t_0<t_{l_1}<\cdots<t_{l_5}<t_l<\cdots<t_{l+p}$, $x_i\sim\mathrm{U}(0,1),~\alpha_j\sim\mathrm{U}(0,1),~  \epsilon_i\sim\mathrm{N}(0,0.1^2)$, and they are mutually independent.

As for the setting that changes $n$, the true used knots $t_{l_1}<\ldots<t_{l_5}$ was uniformly randomly drawn from $\{t_1,\ldots,t_{l-1}\}$.
Figure \ref{fig:vs-a-spline-1} shows box plots of the natural logarithm of the mean squared error (MSE) of the model obtained by each method for several $n$.
Here, the mean squared error
\begin{align}
    \int_0^1 \big(\Tilde{s}(x)-\hat{s}(x)\big)^2 dx
\end{align}
was computed by the Monte Carlo method with sample size $10^5$, where $\Tilde{s}$ and $\hat{s}$ are the true model and estimated model, respectively.
The A-spline failed to estimate when the sample size $n$ is small.
On the other hand, the proposed methodology was able to estimate stably even when the sample size is small.

As for the setting that changes $l$, the true used knots $t_{l_1}<\ldots<t_{l_5}$ were determined by sorting those independently drawn from $\mathrm{U}(0,1)$.
Figure \ref{fig:vs-a-spline-2} shows box plots of the natural logarithm of the MSE of the model obtained by each method for several $l$.
The frequency of estimation failures for the A-spline increased as the number of candidates $l$ increases.
Our methodology succeeded in stable estimation without the influence of the number of candidates.
This implies that the proposed methodology enables us to select position of knots almost continuously.

Next, we consider model selection based on the BIC for the fossil data ($n=106$).
Figure \ref{fig:vs-a-spline-3} shows that the proposed method estimated a simpler model than the A-spline. 
The selected locations of knots are indicated in the figure by black diamonds $\blacklozenge$, and we see that our method appears to have chosen fewer knots than the A-spline, which seems to overfit the data over some intervals where the knots are densely chosen. 
Figure \ref{fig:vs-a-spline-4} shows the case of four times as many knots candidates as in Figure \ref{fig:vs-a-spline-3}. 
With more knots candidates, the A-spline showed violent fluctuations, resulting in poor estimates, but our methodology estimated a similar simple model as with fewer knots candidates.

\begin{figure}[t]
    \centering
    \includegraphics[width=0.8\columnwidth]{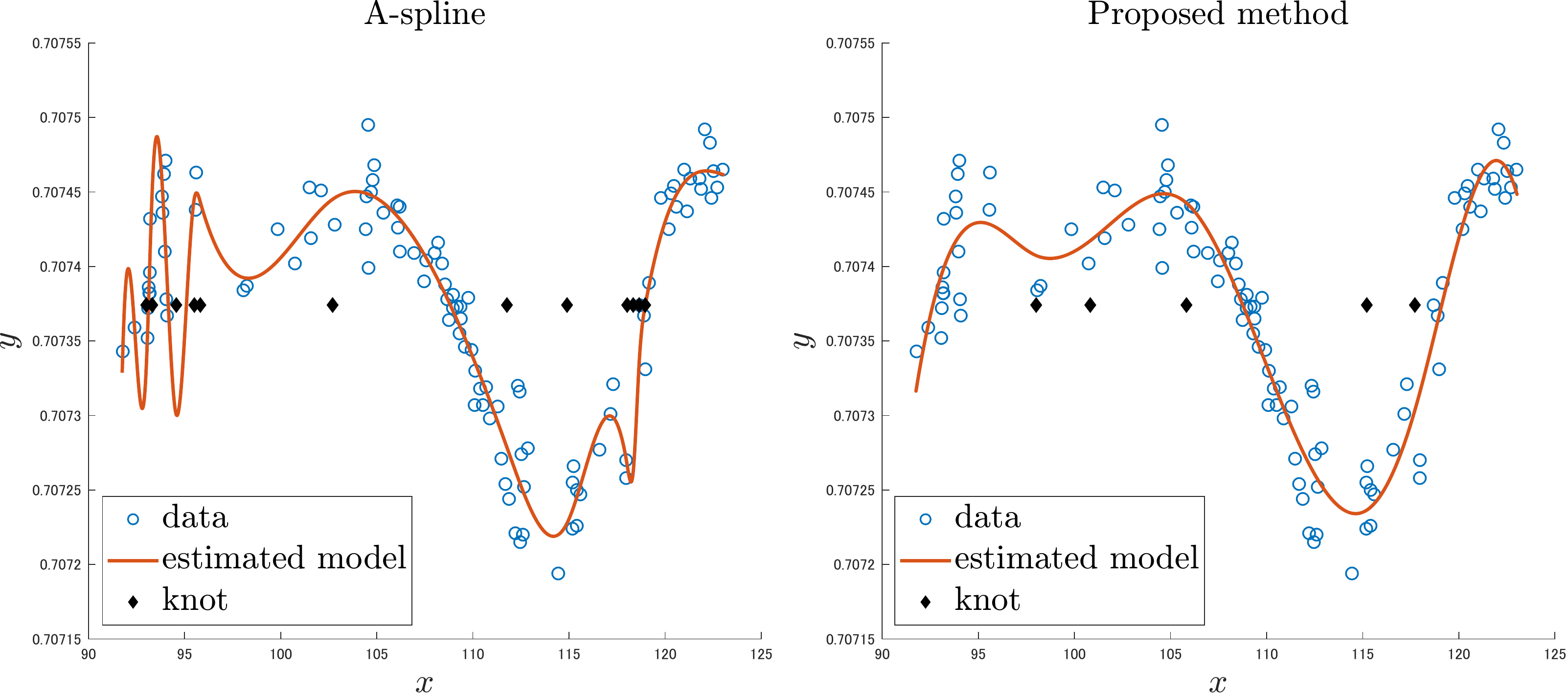}
    \caption{Fossil data and estimated model with $l=100$}
    \label{fig:vs-a-spline-3}
    \vspace{20pt}
    \includegraphics[width=0.8\columnwidth]{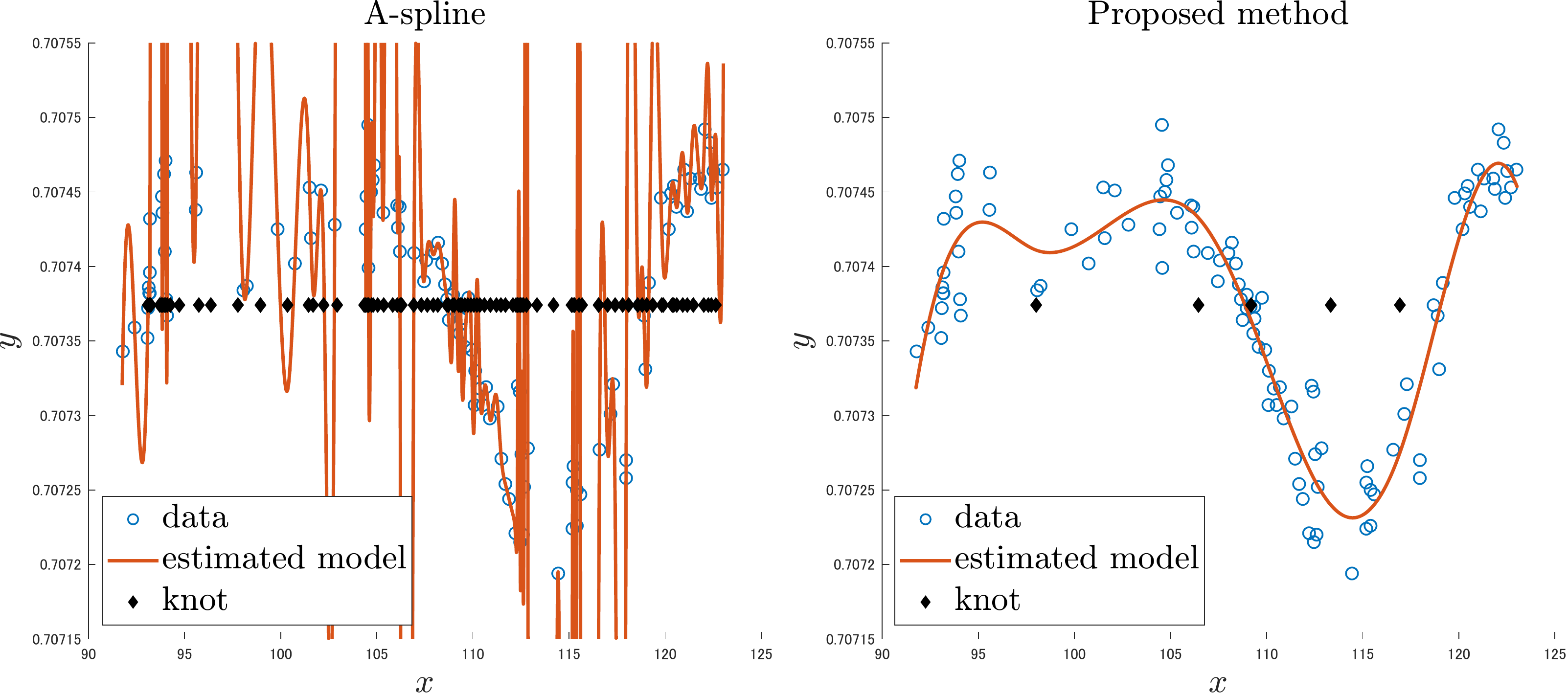}
    \caption{Fossil data and estimated model with $l=400$}
    \label{fig:vs-a-spline-4}
\end{figure}

\section{Concluding Remarks}\label{sec:conclusion}
A new continuous optimization methodology for automatic selection of the locations of the knots of the B-spline is presented in this paper. 
We first show the corrected proposition for selecting knots using the difference matrix of arbitrary order (Theorem \ref{thm:active_knot_indicator}).
Based on it, we consider the formulation \eqref{obj:trimmed_penalized} as an example of the generalized trimmed lasso \citep{yagishita2022exact}. 
The equivalence is established by showing an exact penalty which can be computed easily under a mild assumption (Theorem \ref{thm:e.p.p.}).
It is proposed to apply the change of variables and reducing variables to \eqref{obj:trimmed_penalized} and solve by the GIST \citep{wright2009sparse,gong2013general}.
Our methodology, namely, the combination of the formulation with the exact penalty and the algorithm enables the user to obtain a spline regression model using no greater than the designated number of knots since we show that the GIST converges to a local optimum, which satisfies the cardinality constraint (Theorem \ref{thm:global-convergence}). 
Numerical experiments indicate the effectiveness of our proposed method. 
The proposed method may perform better if the choice of the roughness penalty parameter $c$ is well taken into account, which we leave for future work.

% \appendix
\begin{appendices}

\section{Proofs of propositions}
The proofs omitted in the main body of the paper are displayed in the following.

\subsection{Proof of Theorem \ref{thm:active_knot_indicator}}\label{subsec:proof_active_knot_indicator}
To prove Theorem \ref{thm:active_knot_indicator} we start with the following lemma.

\begin{lemma}\label{lem:max-cofficient}
For $p\ge0$, let $s=\sum_{j=-p}^{l-1}\alpha_{j+p+1}B^{(p)}_j,~ \bm{\alpha}=(\alpha_1,\dots,\alpha_{p+l})^\top$ and $i\in\{1,\ldots,l-1\}$.
The coefficient of the $p$-th order term of $s$ over the interval $(t_i,t_{i+1})$ is denoted by the $i$-th element of the vector $\bm{D}^{(p+1)}_+\bm{\alpha}$, and that over the interval $(t_{i-1},t_{i})$ is denoted by the $i$-th element of the vector $\bm{D}^{(p+1)}_-\bm{\alpha}$.
\end{lemma}

\begin{proof}
We will prove by induction.
For $p=0$, we see that
\begin{align}
s &=\sum_{j=0}^{l-1}\alpha_{j+1}\bm{1}_{[t_j,t_{j+1})}\\
&=\sum_{j=1}^{l}\alpha_{j}\bm{1}_{[t_{j-1},t_j)},
\end{align}
$(\bm{D}^{(1)}_+\bm{\alpha})_i=\alpha_{i+1}$, and $(\bm{D}^{(1)}_-\bm{\alpha})_i=\alpha_i$.
Thus the statement holds for $p=0$.

Next, we assume that the statement holds for $p-1$.
We obtain from the definition of $B^{(p)}_j$ that
\begin{align}
s(x) &=\sum_{j=-p}^{l-1}\alpha_{j+p+1}B^{(p)}_j(x)\\
&=\sum_{j=-p}^{l-1}\alpha_{j+p+1}\left\{\frac{x-t_{j}}{t_{j+p}-t_{j}}B^{(p-1)}_j(x)+\frac{t_{j+p+1}-x}{t_{j+p+1}-t_{j+1}}B^{(p-1)}_{j+1}(x)\right\}\\
&=\sum_{j=-(p-1)}^{l-1}\left\{\alpha_{j+p+1}\frac{x-t_{j}}{t_{j+p}-t_{j}}+\alpha_{j+p}\frac{t_{j+p}-x}{t_{j+p}-t_j}\right\}B^{(p-1)}_j(x)\\
&=\sum_{j=-(p-1)}^{l-1}\left\{\frac{\alpha_{j+p+1}-\alpha_{j+p}}{t_{j+p}-t_j}x-\frac{\alpha_{j+p+1}t_j-\alpha_{j+p}t_{j+p}}{t_{j+p}-t_j}\right\}B^{(p-1)}_j(x)\\
&=\sum_{j=-(p-1)}^{l-1}\left\{(\bm{\Delta}^{(p+1)}\bm{\alpha})_{j+p}x-\frac{\alpha_{j+p+1}t_j-\alpha_{j+p}t_{j+p}}{t_{j+p}-t_j}\right\}B^{(p-1)}_j(x)\\
&=x\sum_{j=-(p-1)}^{l-1}(\bm{\Delta}^{(p+1)}\bm{\alpha})_{j+p}B^{(p-1)}_j(x)-\sum_{j=-(p-1)}^{l-1}\left\{\frac{\alpha_{j+p+1}t_j-\alpha_{j+p}t_{j+p}}{t_{j+p}-t_j}\right\}B^{(p-1)}_j(x)
\end{align}
for all $x\in[t_0,t_l)$, where the third equality follows from $B^{(q-1)}_{-q}=B^{(q-1)}_l=0$ on $[t_0,t_l)$.
Using the induction hypothesis, we see that the coefficient of the $(p-1)$-th order term of $s$ on the interval $(t_i,t_{i+1})$ is $(\bm{D}^{(p)}_+\bm{\Delta}^{(p+1)}\bm{\alpha})_i=(\bm{D}^{(p+1)}_+\bm{\alpha})_i$, and that on the interval $(t_{i-1},t_{i})$ is $(\bm{D}^{(p)}_-\bm{\Delta}^{(p+1)}\bm{\alpha})_i=(\bm{D}^{(p+1)}_-\bm{\alpha})_i$.
This completes the proof.
\end{proof}

\begin{proof}[proof of Theorem \ref{thm:active_knot_indicator}]
Note that $B^{(p)}_{-p},\ldots,B^{(p)}_{l-1}$ restricted to $[t_0,t_l)$ is a basis of the linear space consisting of piece-wise polynomials of order $p$ on $[t_0,t_l)$ with breakpoints $t_1,\ldots,t_{l-1}$ whose derivatives coincide up to order $p-1$ at all the breakpoints \citep[pp.97--98]{de1978practical}.
Thus the function $s$ does not use the $i$-th knot if and only if the coefficient of the $p$-th order term of $s$ over the interval $(t_i,t_{i+1})$ coincides with that over the interval $(t_{i-1},t_{i})$.
Since it holds that $\bm{D}^{(p+1)}=\bm{D}^{(p+1)}_+-\bm{D}^{(p+1)}_-$, $(\bm{D}^{(p+1)}\bm{\alpha})_{i}=0$ is equivalent to $(\bm{D}^{(p+1)}_+\bm{\alpha})_i=(\bm{D}^{(p+1)}_-\bm{\alpha})_i$.
Therefore, we have the desired result from Lemma \ref{lem:max-cofficient}.
\end{proof}

\subsection{Proof of Proposition \ref{prop:relation}}\label{subsec:proof_relation}
\begin{proof}
Let $\bm{\beta}^*$ be a local minimum of \eqref{obj:trimmed_penalized_TF2}, that is, there exists a neighborhood $\mathcal{N}$ of $\bm{\beta}^*$ such that $F(\bm{\beta}^*)\le F(\bm{\beta})$ holds for any $\bm{\beta}\in\mathcal{N}$.
Noting that $F(\bm{\beta})=G(\bm{\beta},g(\bm{\beta}))$ holds for any $\bm{\beta}\in\mathbb{R}^{l-1}$, we have
\begin{align}
    G(\bm{\beta}^*,g(\bm{\beta}^*))=F(\bm{\beta}^*)\le F(\bm{\beta})\le G(\bm{\beta},\bm{\beta}')
\end{align}
for any $\bm{\beta}\in\mathcal{N}$ and $\bm{\beta}'\in\mathbb{R}^{p+1}$, which implies that $(\bm{\beta}^*,g(\bm{\beta}^*))$ is locally optimal to \eqref{obj:trimmed_penalized_TF}.
It is clear that the local optimality of $(\bm{\beta}^*,g(\bm{\beta}^*))$ to \eqref{obj:trimmed_penalized_TF} and $\bm{\Sigma}^{(p+1)}(\bm{\beta}^{*\top},g(\bm{\beta}^*)^\top)^\top$ to \eqref{obj:trimmed_penalized} are equivalent.
This completes the proof of the former argument.
The latter claim can be proved as well.
\end{proof}

\subsection{Proof of Theorem \ref{thm:e.p.p.}}\label{subsec:proof_e.p.p.}
\begin{proof}
Let $h(\bm{\beta})\coloneqq\frac{1}{2}\big\|\bm{z}_1-\bm{L}_1\bm{\beta}\big\|_2^2+\frac{c}{2}\big\|\bm{z}_2-\bm{L}_2\bm{\beta}\big\|_2^2$.
We see from the d-stationarity of $\bm{\beta}^*$ that
\begin{align}\label{eq:d-stat-to-0}
\begin{split}
    &\nabla h(\bm{\beta}^*)^\top(0-\bm{\beta}^*)+\gamma T_K'(\bm{\beta}^*;-\bm{\beta}^*)\\
    &=h'(\bm{\beta}^*;-\bm{\beta}^*)+\gamma T_K'(\bm{\beta}^*;-\bm{\beta}^*)\\
    &=F'(\bm{\beta}^*;-\bm{\beta}^*)\\
    &\ge0.
\end{split}
\end{align}
Since it is easy to see that $T_K'(\bm{\beta}^*;-\bm{\beta}^*)=-T_K(\bm{\beta}^*)$, we have
\begin{align}
    \|\bm{z}_1\|_2^2+c\|\bm{z}_2\|_2^2\ge\big\|\bm{z}_1-\bm{L}_1\bm{\beta}^*\big\|_2^2+c\big\|\bm{z}_2-\bm{L}_2\bm{\beta}^*\big\|_2^2
\end{align}
from \eqref{eq:d-stat-to-0}, the convexity of $h$, and the non-negativity of $T_K$.
This leads to
\begin{align}
    \big\|\bm{z}_1-\bm{L}_1\bm{\beta}^*\big\|_2\le\sqrt{\|\bm{z}_1\|_2^2+c\|\bm{z}_2\|_2^2}, \quad \sqrt{c}\big\|\bm{z}_2-\bm{L}_2\bm{\beta}^*\big\|_2\le\sqrt{\|\bm{z}_1\|_2^2+c\|\bm{z}_2\|_2^2}
\end{align}
and hence we can evaluate as
\begin{align}
    \nabla h(\bm{\beta}^*)^\top\bm{d} &\le\|\nabla h(\bm{\beta}^*)\|_\infty\\
    &=\max_{j=1,\ldots,l-1}\left|{\bm{l}^{(1)}_j}^\top(\bm{z}_1-\bm{L}_1\bm{\beta}^*)+c{\bm{l}^{(2)}_j}^\top(\bm{z}_2-\bm{L}_2\bm{\beta}^*)\right|\\
    &\le\max_{j=1,\ldots,l-1}\left\{\|\bm{l}^{(1)}_j\|_2\|\bm{z}_1-\bm{L}_1\bm{\beta}^*\|_2+c\|\bm{l}^{(2)}_j\|_2\|\bm{z}_2-\bm{L}_2\bm{\beta}^*\|_2\right\}\\
    &\le\max_{j=1,\ldots,l-1}\left\{\|\bm{l}^{(1)}_j\|_2\sqrt{\|\bm{z}_1\|_2^2+c\|\bm{z}_2\|_2^2}+\sqrt{c}\|\bm{l}^{(2)}_j\|_2\sqrt{\|\bm{z}_1\|_2^2+c\|\bm{z}_2\|_2^2}\right\}\\
    &=\max_{j=1,\ldots,l-1}\left\{\|\bm{l}^{(1)}_j\|_2+\sqrt{c}\|\bm{l}^{(2)}_j\|_2\right\}\sqrt{\|\bm{z}_1\|_2^2+c\|\bm{z}_2\|_2^2}.
\end{align}
for all $d$ such that $\|d\|_1=1$ and $d\in\{-1,0,1\}^{l-1}$.
From Lemma 4 of \citet{yagishita2022exact}, we obtain $\|\bm{\beta}^*\|_0\le K$.
Noting that
\begin{align}
    \begin{pmatrix}
        \bm{D}^{(p+1)}\bm{\Sigma}^{(p+1)}(\bm{\beta}^{*\top},g(\bm{\beta}^*)^\top)^\top\\
        \bm{A}\bm{\Sigma}^{(p+1)}(\bm{\beta}^{*\top},g(\bm{\beta}^*)^\top)^\top
    \end{pmatrix}
    =
    \hat{\bm{D}}^{(p+1)}\bm{\Sigma}^{(p+1)}
    \begin{pmatrix}
        \bm{\beta}^*\\
        g(\bm{\beta}^*)
    \end{pmatrix}
    =
    \begin{pmatrix}
        \bm{\beta}^*\\
        g(\bm{\beta}^*)
    \end{pmatrix}
\end{align}
and that $\bm{\Sigma}^{(p+1)}(\bm{\beta}^{*\top},g(\bm{\beta}^*)^\top)^\top$ is local minimum of \eqref{obj:trimmed_penalized} according to Proposition \ref{prop:relation}, we have the desired result.
\end{proof}

\subsection{Proof of Theorem \ref{thm:global-convergence}}\label{subsec:proof_global-convergence}
\begin{proof}
Let $\Omega\coloneqq\{\bm{\beta}\mid F(\bm{\beta})\le F(\bm{\beta}_0)\}$.
Note that $\{\bm{\beta}_t\}\subset\Omega$.
From the non-negativity of $T_K$, it holds that $h(\bm{\beta})\le F(\bm{\beta}_0)$ for any $\bm{\beta}\in\Omega$, which leads to
\begin{align}
    \big\|\bm{z}_1-\bm{L}_1\bm{\beta}\big\|_2\le\sqrt{2F(\bm{\beta}_0)}, \quad \big\|\bm{z}_2-\bm{L}_2\bm{\beta}\big\|_2\le\sqrt{\frac{2F(\bm{\beta}_0)}{c}}.
\end{align}
Thus, we have
\begin{align}
    \|\nabla h(\bm{\beta})\|_2 &=\left\|\bm{L}_1^\top(\bm{z}_1-\bm{L}_1\bm{\beta})+c\bm{L}_2^\top(\bm{z}_2-\bm{L}_2\bm{\beta})\right\|_2\\
    &\le\|\bm{L}_1\|\left\|\bm{z}_1-\bm{L}_1\bm{\beta}\right\|_2+c\|\bm{L}_2\|\left\|\bm{z}_2-\bm{L}_2\bm{\beta}\right\|_2\\
    &\le\|\bm{L}_1\|\sqrt{2F(\bm{\beta}_0)}+\|\bm{L}_2\|\sqrt{2cF(\bm{\beta}_0)}
\end{align}
for any $\bm{\beta}\in\Omega$ where $\|\cdot\|$ is the operator norm, which implies that $h$ is Lipschitz continuous on $\Omega$.
Furthermore, $T_K$ is also Lipschitz continuous because it is expressed as the difference between the $\ell_1$ norm and the largest-$K$ norm.
As a result, $F$ is Lipschitz continuous on $\Omega$, namely, is also uniformly continuous on $\Omega$.
Combining the uniform continuity and non-negativity of $F$ with the Lipschitz continuity of $\nabla h$ yields $\|\bm{\beta}_{t+1}-\bm{\beta}_t\|_2\to0$, similarly to the proof of Lemma 4 of \citet{wright2009sparse}.
Let $\bm{\beta}^*$ be an accumulation point of $\{\bm{\beta}_t\}$ and $\{\bm{\beta}_{t_i}\}$ be a subsequence that converges to $\bm{\beta}^*$.
Since it is easy to see that $\{\eta_t\}$ is bounded (see, for example, \citet[Theorem 5.1]{lu2018sparse}), for any $\bm{d}\in\mathbb{R}^{l-1}$, it follows from the optimality of $\bm{\beta}_{t_i+1}$ that
\begin{align}
    &\nabla h(\bm{\beta}_{t_i})^\top\bm{\beta}_{t_i+1}+\frac{\eta_{t_i}}{2}\|\bm{\beta}_{t_i+1}-\bm{\beta}_{t_i}\|_2^2+T_K(\bm{\beta}_{t_i+1})\\
    &\le\nabla h(\bm{\beta}_{t_i})^\top(\bm{\beta}^*+\xi\bm{d})+\frac{\eta_{t_i}}{2}\|\bm{\beta}^*+\xi\bm{d}-\bm{\beta}_{t_i}\|_2^2+T_K(\bm{\beta}^*+\xi\bm{d})\\
    &\le\nabla h(\bm{\beta}_{t_i})^\top(\bm{\beta}^*+\xi\bm{d})+\frac{\eta_{\max}}{2}\|\bm{\beta}^*+\xi\bm{d}-\bm{\beta}_{t_i}\|_2^2+T_K(\bm{\beta}^*+\xi\bm{d})
\end{align}
for $\xi>0$, where $\eta_{\max}\coloneqq\sup_t\eta_t$.
We obtain from the continuity of $\nabla h$ and $T_K$ that
\begin{align}
    \xi\nabla h(\bm{\beta}^*)^\top\bm{d}+\frac{\eta_{\max}\xi^2}{2}\|\bm{d}\|_2^2+T_K(\bm{\beta}^*+\xi\bm{d})-T_K(\bm{\beta}^*)\ge0.
\end{align}
Dividing both sides by $\xi$ and taking the limit $\xi\to0$ give
\begin{align}
    F'(\bm{\beta}^*;\bm{d})=\nabla h(\bm{\beta}^*)^\top d+\gamma T_K'(\bm{\beta}^*;\bm{d})\ge0,
\end{align}
which implies that $\bm{\beta}^*$ is a d-stationary point of \eqref{obj:trimmed_penalized_TF2}, that is, a local minimum of \eqref{obj:trimmed_penalized_TF2}.

Next, let $\bm{\alpha}^*$ be an accumulation point of $\big\{\bm{\Sigma}^{(p+1)}(\bm{\beta}_t^\top,g(\bm{\beta}_t)^\top)^\top\big\}$ and $\big\{\bm{\Sigma}^{(p+1)}(\bm{\beta}_{t_i}^\top,g(\bm{\beta}_{t_i})^\top)^\top\big\}$ be a subsequence that converges to $\bm{\alpha}^*$.
We see that
\begin{align}\label{eq:relation}
\begin{split}
    \begin{pmatrix}
        \bm{\beta}_{t_i}\\
        g(\bm{\beta}_{t_i})
    \end{pmatrix}
    =\hat{\bm{D}}^{(p+1)}\bm{\Sigma}^{(p+1)}
    \begin{pmatrix}
        \bm{\beta}_{t_i}\\
        g(\bm{\beta}_{t_i})
    \end{pmatrix}
    \to\hat{\bm{D}}^{(p+1)}\bm{\alpha}^*=
    \begin{pmatrix}
        \bm{D}^{(p+1)}\bm{\alpha}^*\\
        \bm{A}\bm{\alpha}^*
    \end{pmatrix},
\end{split}
\end{align}
which implies that $\bm{D}^{(p+1)}\bm{\alpha}^*$ is a local minimum of \eqref{obj:trimmed_penalized_TF2} because it is an accumulation point of $\{\bm{\beta}_t\}$.
It follows from \eqref{eq:relation} and the continuity of $g$ that $g(\bm{D}^{(p+1)}\bm{\alpha}^*)=\bm{A}\bm{\alpha}^*$ and hence we obtain from Proposition \ref{prop:relation} that
\begin{align}
    \bm{\alpha}^*=\bm{\Sigma}^{(p+1)}\hat{\bm{D}}^{(p+1)}\bm{\alpha}^*=\bm{\Sigma}^{(p+1)}
    \begin{pmatrix}
        \bm{D}^{(p+1)}\bm{\alpha}^*\\
        g(\bm{D}^{(p+1)}\bm{\alpha}^*)
    \end{pmatrix}
\end{align}
is locally optimal to \eqref{obj:trimmed_penalized}.
\end{proof}

\section{Construction of $\hat{\bm{D}}^{(p+1)}$ and $\bm{\Sigma}^{(p+1)}$}\label{sec:diff-mat}
In this section, concrete constructions of $\hat{\bm{D}}^{(p+1)}$ and $\bm{\Sigma}^{(p+1)}$ are shown.
Let us define
\begin{align}
    \hat{\bm{D}}^{(1)}\coloneqq
    \begin{pmatrix}
        -1 & 1 & & & & & & \\
         & \ddots & \ddots & & & & & \\
         & & -1 & 1 & & & \\
         & & & s_1 & & & \\
         & & & & & \ddots & \\
         & & & & & & s_1 \\
    \end{pmatrix}
    \in\mathbb{R}^{(l+p)\times(l+p)}
\end{align}
by expanding $\bm{D}^{(1)}$ with $s_1\neq0$ and
\begin{align}
    \hat{\bm{\Delta}}^{(q+1)}\coloneqq
    \begin{pmatrix}
        \frac{-1}{t_1-t_{-q+1}} & \frac{1}{t_1-t_{-q+1}} &  & & \\
         & \frac{-1}{t_2-t_{-q+2}} & \frac{1}{t_2-t_{-q+2}} & & \\
         & & \ddots & \ddots & \\
         & & & \frac{-1}{t_{l-1+q}-t_{l-1}} & \frac{1}{t_{l-1+q}-t_{l-1}} \\
         & & & & s_{q+1} & & \\
         & & & & & \ddots & \\
         & & & & & & s_{q+1} \\
    \end{pmatrix}
    \in\mathbb{R}^{(l+p)\times(l+p)}
\end{align}
by expanding $\bm{\Delta}^{(q+1)}$ with $s_{q+1}\neq0$ for $1\le q\le p$.
Let
\begin{align}
\hat{\bm{D}}^{(q+1)}\coloneqq\hat{\bm{D}}^{(q)}\hat{\bm{\Delta}}^{(q+1)}
\end{align}
recursively, then there exists a $(p+1)\times(l+p)$ matrix $\bm{A}$ such that
\begin{align}
    \hat{\bm{D}}^{(p+1)}=
    \begin{pmatrix}
        \bm{D}^{(p+1)}\\
        \bm{A}
    \end{pmatrix}.
\end{align}
By constructions of $\hat{\bm{D}}^{(1)}$ and $\hat{\bm{\Delta}}^{(q+1)}$, they are non-singular, and hence $\hat{\bm{D}}^{(p+1)}$ is also non-singular.
It is not hard to see that
\begin{align}\label{eq:inverse-diff-1}
    \big(\hat{\bm{D}}^{(1)}\big)^{-1}=
    \begin{pmatrix}
        -\bm{U} & \bm{S} \\
         & s_{1}^{-1}\bm{I}_{p+1}
    \end{pmatrix}
\end{align}
and
\begin{align}\label{eq:inverse-diff-2}
    \big(\hat{\bm{\Delta}}^{(q+1)}\big)^{-1}=
    \begin{pmatrix}
        -(t_1-t_{-q+1}) & -(t_2-t_{-q+2}) & \cdots & -(t_{l-1+q}-t_{l-1}) & s_{q+1}^{-1} \\
         & -(t_2-t_{-q+2}) & \cdots & -(t_{l-1+q}-t_{l-1}) & s_{q+1}^{-1} \\
         & & \ddots & \vdots & \vdots \\
         & & & -(t_{l-1+q}-t_{l-1}) & \vdots \\
         & & & & s_{q+1}^{-1} & & \\
         & & & & & \ddots & \\
         & & & & & & s_{q+1}^{-1} \\
    \end{pmatrix}
\end{align}
for $1\le q\le p$, where $\bm{U}$ is the upper triangular matrix of size $(l-1)\times(l-1)$ such that all non-zero elements equal $1$ and
\begin{align}
    \bm{S}=
    \begin{pmatrix}
        s_{1}^{-1} & 0 & \ldots & 0 \\
        \vdots & \vdots & & \vdots \\
        s_{1}^{-1} & 0 & \ldots & 0 \\
    \end{pmatrix}
    \in\mathbb{R}^{l-1)\times(p+1)}.
\end{align}
As a result, we can compute as
\begin{align}
    \big(\bm{\Sigma}^{(p+1)}\big)^{-1}=\big(\hat{\bm{D}}^{(1)}\big)^{-1}\cdots \big(\hat{\bm{\Delta}}^{(p+1)}\big)^{-1}
\end{align}
by using \eqref{eq:inverse-diff-1} and \eqref{eq:inverse-diff-2}.

\end{appendices}

\bibliographystyle{plainnat}
\bibliography{reference.bib}

\end{document}